\documentclass[11pt,reqno]{amsart}

\usepackage{color, MnSymbol, bm}
\usepackage{multirow}

\usepackage{natbib}
\usepackage{hyperref}
\usepackage[latin1]{inputenc}
\usepackage{graphicx}
\usepackage{bbm}
\usepackage{subcaption}
\usepackage{float}

\numberwithin{equation}{section}

\theoremstyle{plain}                    
\newtheorem{thm}{Theorem}[section]
\newtheorem{lem}{Lemma}[section]
\newtheorem{cor}{Corollary}[section]
\newtheorem{prop}{Proposition}[section]

\theoremstyle{definition}

\theoremstyle{remark}
\newtheorem{rem}{Remark}               

\def\N{{\mathbb N}}
\def\Z{{\mathbb Z}}
\def\R{{\mathbb R}}
\def\P{{\mathbb P}}
\def\E{{\mathbb E}}
\def\F{{\mathcal F}}
\def\1{{\mathbbm{1}}}
\def\wtl{{\widetilde{\lambda}}}
\def\wty{{\widetilde{Y}}}
\def\wtx{{\widetilde{X}}}
\def\wtz{{\widetilde{Z}}}
\def\ninfty{\mathop{\longrightarrow}\limits_{n\to\infty}}

\newcommand{\var}{\mathop{\rm var}\nolimits}
\newcommand{\cov}{\mathop{\rm cov}\nolimits}

\newcommand{\be}{\begin{equation}}
\newcommand{\bd}{\begin{displaymath}}
\newcommand{\ed}{\end{displaymath}}
\newcommand{\bea}{\begin{eqnarray}}
\newcommand{\eea}{\end{eqnarray}}
\newcommand{\bean}{\begin{eqnarray*}}
\newcommand{\eean}{\end{eqnarray*}}

\begin{document}

\thispagestyle{empty}

\baselineskip13pt

\begin{center}
{\large \sc Mixing properties of non-stationary INGARCH(1,1) processes}
\end{center}

\vspace*{1cm}

\begin{center}
Paul Doukhan\\
CY  University\\
UMR 8088 Analyse, G\'eom\'etrie et 
Mod\'elisation\\
2, avenue Adolphe Chauvin\\
95302 Cergy-Pontoise Cedex\\
France\\
E-mail: doukhan@cyu.fr\\[1cm]
Michael H.~Neumann\\
Friedrich-Schiller-Universit\"at Jena\\
Institut f\"ur Mathematik\\
Ernst-Abbe-Platz 2\\
D -- 07743 Jena\\
Germany\\
E-mail: michael.neumann@uni-jena.de\\[1cm]
Anne Leucht\\
Universit\"at Bamberg\\
Research group of Statistics and Mathematics\\
Feldkirchenstra\ss e 21\\
D -- 96052 Bamberg\\
Germany\\
E-mail: anne.leucht@uni-bamberg.de\\[1.5cm]
\end{center}

\vspace*{1.0cm}

\begin{center}
{\bf Abstract}\end{center}
We derive mixing properties for a broad class of Poisson count time series satisfying a certain
contraction condition. Using specific coupling techniques, we prove absolute regularity at a geometric rate
not only for stationary Poisson-GARCH processes but also for models with an explosive trend.
We provide easily verifiable sufficient conditions for absolute regularity for a variety of models
including classical (log-)linear models. 
Finally, we illustrate the practical use of our results for hypothesis testing.

\vspace*{1cm}

\footnoterule \noindent {\sl 2010 Mathematics Subject
Classification:} Primary 60G10; secondary 60J05. 
{\sl Short title:} Mixing of non-stationary INGARCH processes. \vfill
\noindent
version: \today

\newpage

\setcounter{page}{1}
\pagestyle{headings}
\normalsize

\section{Introduction}
\label{S1}

Conditional heteroscedastic processes have become quite popular for modeling the evolution
of stock prices, exchange rates and interest rates. Starting with the seminal papers
by \citet{Eng82} on autoregressive conditional heteroscedastic models (ARCH)
and \citet{Bol86} on generalized ARCH, numerous variants of these models have been
proposed for modeling financial time series; see for example \citet{FZ10} for
a detailed overview.
More recently, integer-valued GARCH models (INGARCH) which mirror the structure 
of GARCH models have been proposed for modeling time series of counts;
see for example \citet{Fok12} and the recently edited volume by \citet{DHLR16}.

We consider integer-valued processes where the count variable~$Y_t$ at time~$t$, given the past,
has a Poisson distribution with intensity~$\lambda_t$. The intensity~$\lambda_t$ itself
is random and it is assumed that $\lambda_t=f_t(Y_{t-1},\lambda_{t-1},Z_{t-1})$, for some
function~$f_t$, i.e.~$\lambda_t$ is a function of lagged values of the count and intensity processes
and a covariate~$Z_{t-1}$.
Mixing properties
of such processes have been derived for a first time in \citet{Neu11}, for a time-homogeneous
transition mechanism with $\lambda_t=f(Y_{t-1},\lambda_{t-1})$.
This has been generalized by \citet{Neu21} to a GARCH structure of arbitrary order.
In both cases a contractive condition on the intensity function~$f$ was imposed which resulted
in an exponential decay of the coefficients of absolute regularity.
Under a weaker semi-contractive rather than a fully contractive condition on the intensity function,
\citet{DN19} also proved absolute regularity of the count process, this time with a slower
subexponential decay of the mixing coefficients.
In the present paper we extend these results in two directions. We include an exogeneous
covariate process in the intensity function and we also drop the condition of time-homogeneity.
This allows us to consider ``weakly non-stationary'' processes, e.g.~with a periodic pattern
in the intensity function. Moreover, we also allow for a certain explosive behavior which could 
e.g.~result from a deterministic trend. As shown in the text, this requires certain
modifications of the techniques used in our previous work.

In the next section, we state the precise conditions, describe our approach of deriving mixing
properties, and state the main results. In Section~\ref{S3} we apply these results
to time-homogeneous and time-inhomogeneous linear INGARCH models, to
the log-linear model proposed by \citet{FT11} as well as to mixed Poisson INGARCH models.
Section~\ref{S4} clarifies connections to previous work and sketches a few possible extensions. 
In Section~\ref{S5} we discuss a possible application of our results. All proofs are deferred to a final Section~\ref{S6}.
 
\section{Assumptions and main results}
\label{S2}

We derive mixing properties of an integer-valued process ${\mathbf Y}=(Y_t)_{t\in\N_0}$
defined on a probability space $(\Omega,\F,\P)$, where, for $t\geq 1$,
\begin{subequations}
\begin{eqnarray}
\label{2.1a}
Y_t \mid \F_{t-1} \,\sim\, \mbox{Pois}(\lambda_t), \\
\label{2.1b}
\lambda_t \,=\, f_t(Y_{t-1},\lambda_{t-1},Z_{t-1}),
\end{eqnarray}
\end{subequations}
and $\F_s=\sigma(Y_0,\lambda_0,Z_0,\ldots,Y_s,\lambda_s,Z_s)$. Here,
$\mathbf{\lambda}=(\lambda_t)_{t\in\N_0}$ is the process of random (non-negative) intensities
and $\mathbf{Z}=(Z_t)_{t\in\N_0}$ is a sequence of $\R^d$-valued covariates.
We assume that $Z_t$ is independent of~$\F_{t-1}$ and~$Y_t$. We do not assume that the~$Z_t$'s 
are identically distributed since we want to include cases with a possibly unbounded
trend.
Note that with a slight abuse of notation and to avoid an unnecessary case-by-case analysis
Pois(0) denotes the Dirac measure in~$0$.

In what follows we derive conditions which allow us to prove absolute regularity ($\beta$-mixing) 
of the process ${\mathbf X}=(X_t)_{t\in\N_0}$, where $X_t=(Y_t,Z_t)$.
In contrast, the intensity process $(\lambda_t)_{t\in\N_0}$ is not mixing in general;
see Remark~3 in \citet{Neu11} for a counterexample. We will show that, in case of a two-sided stationary process,
$\lambda_t=g(X_{t-1},X_{t-2},\ldots)$, for a suitable function~$g$.
This allows us to conclude that the intensity process, and the joint process $((Y_t,\lambda_t,Z_t))_{t\in\Z}$
as well, are ergodic.

Let $(\Omega,{\mathcal A},P)$ be a probability space and ${\mathcal A}_1$, ${\mathcal A}_2$
be two sub-$\sigma$-algebras of ${\mathcal A}$. Then the coefficient of absolute regularity is defined as
\bd
\beta({\mathcal A}_1,{\mathcal A}_2) \,=\, E\big[ \sup\{ |P(B\mid {\mathcal A}_1) \,-\, P(B)|\colon
\;\; B\in {\mathcal A}_2 \} \big].
\ed
For the process ${\mathbf X}=(X_t)_{t\in\N_0}$ on $(\Omega,\F,\P)$, the
coefficients of absolute regularity at the point~$k$ are defined as
\bd
\beta^X(k,n) \,=\, \beta\big( \sigma(X_0,X_1,\ldots,X_k), \sigma(X_{k+n},X_{k+n+1},\ldots) \big)
\ed
and the (global) coefficients of absolute regularity as
\bd
\beta^X(n) \,=\, \sup\{ \beta^X(k,n)\colon \;\; k\in\N_0\}.
\ed
Our approach of proving absolute regularity is inspired by the fact that one can construct,
on a suitable probability space $(\widetilde{\Omega},\widetilde{F},\widetilde{\P})$,
two versions of the process $\mathbf{X}$, $(\wtx_t)_{t\in\N_0}$ and $(\wtx_t')_{t\in\N_0}$, such that
$(\wtx_0,\ldots,\wtx_k)$ and $(\wtx_0',\ldots,\wtx_k')$ are independent and
\be
\label{mixing-coupling}
\beta^X(k,n) \,=\, \widetilde{\P}\left( \wtx_{k+n+r}\neq\wtx_{k+n+r}' \mbox{ for some } r\geq 0 \right).
\end{equation}
Indeed, for given $(\wtx_t)_{t\in\N_0}$, it follows from Berbee's lemma
(see \citet{Ber79} or \citet[Lemma~5.1]{Rio17}, for a more accessible reference) that
one can construct $(\wtx_t')_{t\geq k+n}$ following the same law as $(\wtx_t)_{t\geq k+n}$
and being independent of $(\wtx_0,\ldots,\wtx_k)$ such that (\ref{mixing-coupling}) is fulfilled.
Using the correct conditional distribution we can augment $(\wtx_t')_{t\geq k+n}$ with $\wtx_0',\ldots,\wtx_{k+n-1}'$
such that $(\wtx_0',\ldots,\wtx_k')$ is independent of $(\wtx_0,\ldots,\wtx_k)$, as required.
Such an ideal coupling is usually hard to find and we do not see a chance to obtain this in the cases
we have in mind. However, any coupling with $(\wtx_0,\ldots,\wtx_k)$ and 
$(\wtx_0',\ldots,\wtx_k')$ being independent provides an estimate of the mixing coefficient
since then 
\bd
\beta^X(k,n) \,\leq\, \widetilde{\P}\left( \wtx_{k+n+r}\neq\wtx_{k+n+r}' \mbox{ for some } r\geq 0 \right);
\ed
see our arguments below.

We obtain the following estimate 
of the coefficients of absolute regularity at the point~$k$.
\bea
\label{2.2}
\lefteqn{ \beta^X(k,n) } \nonumber \\
& = & \beta\big( \sigma(X_0,X_1,\ldots,X_k), \sigma(X_{k+n},X_{k+n+1},\ldots) \big) \nonumber \\
& \leq & \beta\big( \F_k, \sigma(X_{k+n},X_{k+n+1},\ldots) \big) \nonumber \\
& = & \beta\big( \sigma(\lambda_{k+1}), \sigma(X_{k+n},X_{k+n+1},\ldots) \big) \nonumber \\
& = & \E\left[ \sup_{C\in\sigma({\mathcal Z})} \left\{ \Big| \P((X_{k+n},X_{k+n+1},\ldots)\in C\mid \lambda_{k+1})
\,-\, \P((X_{k+n},X_{k+n+1},\ldots)\in C) \Big| \right\} \right], \qquad
\eea
where ${\mathcal Z}=\{A_1\times B_1\times\cdots\times A_m\times B_m\times\N_0\times\R^d\times\N_0\times\R^d\times\cdots\mid
A_1,\ldots,A_m\subseteq\N_0, B_1,\ldots,B_m\in {\mathcal B^d}, m\in\N\}$
is the system of cylinder sets.
Note that the last but one equality in (\ref{2.2}) follows since the process $((Y_t,\lambda_t,Z_t))_{t\in\N_0}$
is Markovian and since the conditional distribution of $(Y_t,\lambda_t,Z_t)$ under $\F_{t-1}$ depends only on
$\lambda_t$.

Since a purely analytic approach to estimate the right-hand side of (\ref{2.2}) seems to be nearly
impossible, we use a {\em stepwise} coupling method to derive the desired result.
Suppose that we have two versions of the original process $((Y_t,\lambda_t,Z_t))_{t\in\N_0}$,
$((\wty_t,\wtl_t,\wtz_t))_{t\in\N_0}$ and $((\wty_t',\wtl_t',\wtz_t'))_{t\in\N_0}$,
which are both defined on the same probability space $(\widetilde{\Omega},\widetilde{\F},\widetilde{\P})$.
If $\wtl_{k+1}$ and $\wtl_{k+1}'$ are independent
under~$\widetilde{\P}$, then we obtain from (\ref{2.2}) the following upper estimate of the coefficients
of absolute regularity at time~$k$:
\bea
\lefteqn{ \beta^X(k,n) } \nonumber \\
& \leq & \widetilde{\E}\left[ \sup_{C\in\sigma({\mathcal Z})} \left\{ \left|
\widetilde{\P}\left( (\wtx_{k+n},\wtx_{k+n+1},\ldots)\in C\mid \wtl_{k+1} \right)
\,-\, \widetilde{\P}\left( (\wtx_{k+n}',\wtx_{k+n+1}',\ldots)\in C\mid \wtl_{k+1}' \right) \right| \right\} \right]
\qquad \nonumber \\
& \leq & \widetilde{\P}\left( \wtx_{k+n+r} \neq \wtx_{k+n+r}' \quad \mbox{for some } r\in\N_0 \right) \nonumber \\
& = & \widetilde{\P}\left( \wtx_{k+n} \neq \wtx_{k+n}' \right) \nonumber \\
& & {} \,+\, \sum_{r=1}^\infty \widetilde{\P}\left( \wtx_{k+n+r} \neq \wtx_{k+n+r}', \wtx_{k+n+r-1} = \wtx_{k+n+r-1}',
\ldots, \wtx_{k+n} = \wtx_{k+n}' \right).\nonumber
\eea
\bigskip

\noindent
Thus we have just proved the following result.
\begin{prop}\label{p.mixing}
If there are two versions,
$((\wty_t,\wtl_t,\wtz_t))_{t\in\N_0}$ and $((\wty_t',\wtl_t',\wtz_t'))_{t\in\N_0}$, of the
process $((Y_t,\lambda_t,Z_t))_{t\in\N_0}$ defined by \eqref{2.1a} and \eqref{2.1b}
which are both defined on the same probability space $(\widetilde{\Omega},\widetilde{\F},\widetilde{\P})$
such that $\wtl_{k+1}$ and $\wtl_{k+1}'$ are independent under~$\widetilde{\P}$, then
\bea
\label{2.3}
\beta^X(k,n) 
& \leq &  \widetilde{\P}\left( \wtx_{k+n} \neq \wtx_{k+n}' \right) \nonumber \\
& & {} \,+\, \sum_{r=1}^\infty \widetilde{\P}\left( \wtx_{k+n+r} \neq \wtx_{k+n+r}', \wtx_{k+n+r-1} = \wtx_{k+n+r-1}',
\ldots, \wtx_{k+n} = \wtx_{k+n}' \right). \quad
\eea
\end{prop}
\bigskip

The close relationship between absolute regularity and coupling has been known for a long time.
\citet[Theorem~2]{Ber79} showed that, for two random variables~$X$ and~$Y$ defined on the same probability space, 
the latter one can be replaced by a random variable~$Y^*$ being independent of~$X$ and following the same distribution
as~$Y$ such that the probability that~$Y^*$ differs from~$Y$ is equal to the coefficient of absolute 
regularity between~$X$ and~$Y$; see also \citet[Theorem~1.2.1.1]{Dou94} for a more accessible reference.
In our paper, we go the opposite way: Starting from a coupling result we derive an upper estimate of the coefficients
of absolute regularity.

In what follows we develop a coupling strategy to keep the right-hand side of (\ref{2.3}) small.
To this end, we couple $\wtz_{k+n+r}$ and $\wtz_{k+n+r}'$ ($r\in\N_0$)
such that they are equal with probability~1, and we apply the technique of {\em maximal coupling}
to the count variables $\wty_{k+n+r}$ and $\wty_{k+n+r}'$.
If $Q_1$ and $Q_2$ are two probability distributions on $(\N_0,2^{\N_0})$, then
one can construct random variables $\bar{X}_1$ and $\bar{X}_2$ on a suitable probability space
$(\bar{\Omega},\bar{\mathcal A},\bar{Q})$ with $\bar{Q}^{\bar{X}_i}=Q_i$, $i=1,2$, such that
\bd
\bar{Q}( \bar{X}_1 \neq \bar{X}_2 ) \,=\, d_{TV}( Q_1, Q_2 ),
\ed
where $d_{TV}(Q_1,Q_2)=\max\{|Q_1(C)-Q_2(C)|\colon \; C\subseteq \N_0\}$
denotes the total variation distance between~$Q_1$ and~$Q_2$.
(An alternative representation is given by $d_{TV}(Q_1,Q_2)=1\,-\,\sum_{k=0}^\infty \min\big\{Q_1(\{k\}),Q_2(\{k\})\big\}$.)
In our case, we have to couple among others $\wty_{k+n}$ and $\wty_{k+n}'$. 
We denote by $\widetilde{\F}_s=\sigma(\wty_0,\wtl_0,\wtz_0,\wty_0',\wtl_0',\wtz_0',\ldots,\wty_s,\wtl_s,\wtz_s,\wty_s',\wtl_s',\wtz_s')$
the $\sigma$-algebra generated by all random variables up to time~$s$.
We construct $\wty_{k+n}$ and $\wty_{k+n}'$ such that,
conditioned on $\widetilde{\F}_{k+n-1}$, they have Poisson distributions with respective
intensities $\wtl_{k+n}$ and $\wtl_{k+n}'$ and
\begin{displaymath}
\widetilde{\P}\left( \wty_{k+n} \neq \wty_{k+n}' \mid \widetilde{\F}_{k+n-1} \right) 
\,=\, d_{TV}\left( \mbox{Pois}(\wtl_{k+n}), \mbox{Pois}(\wtl_{k+n}') \right).
\end{displaymath}
Let $d\colon\; [0,\infty)\times[0,\infty)\rightarrow [0,1]$ be any distance such that
\bd
d_{TV}\left( \mbox{Pois}(\lambda), \mbox{Pois}(\lambda') \right) \,\leq\, d(\lambda, \lambda') 
\qquad \forall \lambda,\lambda'\geq 0.
\ed
Examples for such distances are given by $d(\lambda,\lambda')=\sqrt{2/e}|\sqrt{\lambda}-\sqrt{\lambda'}|$
(see e.g.~\citet[formula (5)]{Roo03} or Exercise~9.3.5(b) in \citet[page~300]{DvJ88}),
and $d(\lambda,\lambda')=|\lambda-\lambda'|$.
Hence, we can construct $\wty_{k+n}$ and $\wty_{k+n}'$ such that
\be
\label{2.4}
\widetilde{\P}\left( \wtx_{k+n} \neq \wtx_{k+n}' \mid \wtl_{k+n}, \wtl_{k+n}' \right)
\,=\, \widetilde{\P}\left( \wty_{k+n} \neq \wty_{k+n}' \mid \wtl_{k+n}, \wtl_{k+n}' \right)
\,\leq\, d\left( \wtl_{k+n}, \wtl_{k+n}' \right).
\end{equation}
Since $Z_t$ is by assumption independent of $\F_{t-1}$ and $Y_t$, we choose $\wtz_{k+n}$ and $\wtz_{k+n}'$
such that they are equal with probability~1.
In view of the other terms on the right-hand side of (\ref{2.3}), we impose the following condition.
\bigskip

\begin{itemize}
\item[{\bf (A1)}]
There exists some $L_1<1$, such that the following condition is fulfilled:
If $\lambda,\lambda'\geq 0$, $Y\sim\mbox{Pois}(\lambda)$ being independent of $Z_t$, then
\bd
\E\left[ d\left( f_t(Y,\lambda,Z_t), f_t(Y,\lambda',Z_t) \right) \right]
\,\leq\, L_1 \; d( \lambda, \lambda' ) \qquad \forall t\in\N.
\ed 
\end{itemize}
\bigskip

Then, if we continue to use maximal coupling,
\bean
\lefteqn{ \widetilde{\P}\left( \wtx_{k+n+1} \neq \wtx_{k+n+1}', \wtx_{k+n} = \wtx_{k+n}' \mid
\wtl_{k+n}, \wtl_{k+n}' \right) } \\
& = & \widetilde{\E}\left( \widetilde{\P}\left( \wtx_{k+n+1} \neq \wtx_{k+n+1}' \mid \widetilde{\F}_{k+n} \right)
\1( \wtx_{k+n} = \wtx_{k+n}' ) \mid \wtl_{k+n}, \wtl_{k+n}' \right) \\
& \leq & \widetilde{\E}\left( d( f_t(\wty_{k+n},\wtl_{k+n},\wtz_{k+n}), f_t(\wty_{k+n}',\wtl_{k+n}',\wtz_{k+n}') ) \;
\1( \wtx_{k+n} = \wtx_{k+n}' ) \Big| \wtl_{k+n}, \wtl_{k+n}' \right) \\
& \leq & \widetilde{\E}\left( d( f_t(\wty_{k+n},\wtl_{k+n},\wtz_{k+n}), f_t(\wty_{k+n},\wtl_{k+n}',\wtz_{k+n}) ) 
\Big| \wtl_{k+n}, \wtl_{k+n}' \right) \\
& \leq & L_1 \; d( \wtl_{k+n}, \wtl_{k+n}' ).
\eean
Proceeding in the same way we obtain that
\bea
\label{2.5}
\lefteqn{ \widetilde{\P}\left( \left. \wtx_{k+n+r} \neq \wtx_{k+n+r}', \wtx_{k+n+r-1} = \wtx_{k+n+r-1}', 
\ldots, \wtx_{k+n} = \wtx_{k+n}' \right| \wtl_{k+n}, \wtl_{k+n}' \right) } \nonumber \\
& \leq & L_1^r \; d( \wtl_{k+n}, \wtl_{k+n}' )
\qquad \qquad \qquad \qquad \qquad \qquad \qquad \qquad \qquad \qquad \qquad
\eea
holds for all $r\in\N$.
It follows from (\ref{2.3}) to (\ref{2.5}) that
\be
\label{2.6}
\beta^X(k,n) \,\leq\, \frac{1}{1-L_1} \;
\widetilde{\E}\left[ d( \wtl_{k+n}, \wtl_{k+n}' ) \right].
\end{equation}

To proceed, we have to find an upper estimate of $\widetilde{\E}[ d(\wtl_{k+n},\wtl_{k+n}') | ]$,
still under the condition that $\wtl_{k+1}$ and $\wtl_{k+1}'$ are independent, having the same distribution as the
frequency $\lambda_{k+1}$ of the original process. We make the following assumption.
\bigskip

\begin{itemize}
\item[{\bf (A2)}]
There exists some $L_2<1$, such that the following condition is fulfilled.
If $\lambda,\lambda'\geq 0$, then there exists a coupling of $(Y,Z)$ and $(Y',Z')$,
with $Y\sim\mbox{Pois}(\lambda)$, $Y'\sim\mbox{Pois}(\lambda')$, $Z,Z'\stackrel{d}{=}Z_t$,
$Z$ being independent of $Y$ and $Z'$ being independent of $Y'$, such that
\bd
\E \left[d\left( f_t(Y,\lambda,Z), f_t(Y',\lambda',Z') \right) \right]
\,\leq\, L_2 \; d( \lambda, \lambda' ) \qquad \forall t\in\N.
\ed 
\end{itemize}
\bigskip

If (A2) is fulfilled, we obtain that
\bd
\widetilde{\E}\left( d( \wtl_{k+n}, \wtl_{k+n}' ) \Big| \wtl_{k+1}, \wtl_{k+1}' \right)
\,\leq\, L_2^{n-1} \; d( \wtl_{k+1}, \wtl_{k+1}' ).
\ed 
Therefore, we obtain in conjunction with (\ref{2.6}) that
\be
\label{2.7}
\beta^X(k,n)
\,\leq\, \frac{1}{1-L_1} \; L_2^{n-1} \; \widetilde{\E}\left[ d( \wtl_{k+1}, \wtl_{k+1}' ) \right].
\end{equation}
Finally, in order to obtain a good bound for $\beta^X(n)$ we have to ensure that
$\sup\{ \widetilde{\E}d(\wtl_{k+1},\wtl_{k+1}')\colon \;\; k\in\N_0\}<\infty$.
Recall that, with the above method of estimating $\beta^X(k,n)$, $\wtl_{k+1}$ and $\wtl_{k+1}'$
have to be independent, following the same distribution as $\lambda_{k+1}$.
In the case of a stationary process, an upper bound may follow from the fact that
the intensities $\lambda_k$ are stochastically bounded in an appropriate sense.
Such an argument, however, cannot be used if the process has an explosive behavior
which means that we genuinely have to derive an upper bound for $\widetilde{\E}d(\wtl_{k+1},\wtl_{k+1}')$,
with an appropriately chosen distance~$d$; see the examples in the next section for the necessity
of a tailor-made way of handling this problem.
As above, it seems to be difficult to derive an upper bound for
$\widetilde{\E}d(\wtl_{k+1},\wtl_{k+1}')$ in a purely analytical way.
Therefore, we employ once more a coupling idea and the desired upper bound will be obtained
by observing two {\em independent} versions $(\wtl_t)_{t\in\N_0}$ and $(\wtl_t')_{t\in\N_0}$
of the original intensity process. We impose the following condition.
\bigskip

\begin{itemize}
\item[{\bf (A3)}]
Let $(\wtl_t)_{t\in\N_0}$ and $(\wtl_t')_{t\in\N_0}$ be two independent processes
on $(\widetilde{\Omega},\widetilde{\F},\widetilde{\P})$ which have the
same distribution as $(\lambda_t)_{t\in\N_0}$. Suppose that there exist constants $L_3<1$ and 
$M_0,M_1<\infty$ such that
\begin{itemize}
\item[(i)] $\quad\widetilde{\E} d(\wtl_0, \wtl_0') \,\leq\, M_0$,
\item[(ii)] $\quad\widetilde{\E}\left( d\left(\wtl_{t+1}, \wtl_{t+1}' \right) \Big| \wtl_t, \wtl_t' \right)
\,\leq\, L_3 \; d( \wtl_t, \wtl_t' ) \,+\, M_1\qquad \forall t\in\N$.
\end{itemize}
\end{itemize}
\bigskip

If (A3) is fulfilled, then
\bean
\widetilde{\E} d\big( \wtl_1, \wtl_1' \big)
& \leq & \widetilde{\E}\left[ \widetilde{\E}\left( d( \wtl_1, \wtl_1' )
\Big| \wtl_0, \wtl_0' \right) \right] \\
& \leq & L_3 \; \widetilde{\E}\left[ d( \wtl_0, \wtl_0' ) \right] \,+\, M_1 \\
& \leq & L_3 \; M_0 \,+\, M_1.
\eean
Furthermore, since $((\wtl_t,\wtl_t'))_{t\in\N_0}$ is a Markov chain,
\bean
\widetilde{\E} d( \wtl_2, \wtl_2')
& \leq & \widetilde{\E}\left[ \widetilde{\E}\left( \widetilde{\E}\left( 
d( \wtl_2, \wtl_2' ) \Big| \wtl_1, \wtl_1' \right) 
\Big| \wtl_0, \wtl_0' \right) \right] \\
& \leq & M_1 \,+\, L_3 \; \left( L_3 \; \widetilde{\E} d( \wtl_0, \wtl_0' )
\,+\, M_1 \right) \\
& \leq & M_1 \; \left( 1 \,+\, L_3 \right) \,+\, L_3^2 \; M_0.
\eean
By induction we obtain that
\be
\label{2.8}
\widetilde{\E} d( \wtl_k, \wtl_k' ) \,\leq\, \frac{M_1}{1 \,-\, L_3} \,+\, M_0
\end{equation}
holds for all $k\in\N$.
Now we obtain from (\ref{2.7}) and (\ref{2.8}) the following result.
\bigskip

{\thm
\label{T2.1}
Suppose that (A1) to (A3) are fulfilled. 
\begin{itemize}
\item[(i)] Then
\bd
\beta^X(n)
\,\leq\, L_2^{n-1} \; \frac{1}{1-L_1} \; \left( \frac{M_1}{1 \,-\, L_3} \,+\, M_0 \right).
\ed
\item[(ii)] Suppose in addition that $((Y_t,\lambda_t,Z_t))_{t\in\Z}$ is a two-sided strictly
stationary version of the process.
Then there exists a $(\sigma({\mathcal Z})-{\mathcal B})$-measurable function~$g$,
where ${\mathcal Z}=\{A_1\times B_1\times\cdots\times A_m\times B_m\times\N_0\times\R^d\times\N_0\times\R^d\times\cdots\mid
A_1,\ldots,A_m\subseteq\N_0, B_1,\ldots,B_m\in {\mathcal B^d}, m\in\N\}$ is the system of cylinder sets,
such that
\be
\label{2.11}
\lambda_t \,=\, g(X_{t-1},X_{t-2},\ldots) \qquad a.s.
\end{equation}
The process $((Y_t,\lambda_t,Z_t))_{t\in\Z}$ is ergodic.
\end{itemize}
}
\bigskip

\begin{rem}
As it can be seen from the proofs of Corollaries~\ref{C3.1} and~\ref{C3.2} below,
the broad applicability of this result is assured by flexibility in the choice of the metric~$d$ in (A1) to (A3);
for details see the discussion about Corollary~\ref{C3.1}. 
\end{rem}
\bigskip

In retrospect, we note that our coupling method which delivers an upper estimate for $\beta^X(n)=\sup\{\beta^X(k,n)\colon\,k\in\N_0\}$
consists of three phases: (\ref{2.7}) shows that the upper estimate depends on the expectation
of $d(\lambda,\lambda')$, where $\lambda$ and $\lambda'$ are independent versions of $\lambda_{k+1}$.
Since this expectation can hardly be computed analytically we consider two {\em independent}
versions, $(\wtl_t)_{t\in\N_0}$ and $(\wtl_t')_{t\in\N_0}$, of the intensity process and we derive
recursively an upper estimate of $\widetilde{\E}d(\wtl_{k+1},\wtl_{k+1}')$.
Condition~(A3) ensures boundedness of this expectation.
Once we have a uniform bound for~$\widetilde{\E}d(\wtl_{k+1},\wtl_{k+1}')$, we start a second coupling mechanism
which keeps the probability of $\wtx_{k+n}\neq \wtx_{k+n}'$ small; see (\ref{2.3}) for how this
enters the upper estimate for $\beta^X(k,n)$. This is accomplished by a coupling which leads to
an exponential decay of $d(\wtl_{k+n},\wtl_{k+n}')$ as $n\to\infty$; (A2) serves this purpose.
And finally, it can also be seen from (\ref{2.3}) that the term
$\sum_{r=1}^\infty \widetilde{\P}\left( \wtx_{k+n+r} \neq \wtx_{k+n+r}', \wtx_{k+n+r-1} = \wtx_{k+n+r-1}',
\ldots, \wtx_{k+n} = \wtx_{k+n}' \right)$ contributes to the upper estimate for $\beta^X(k,n)$.
For this we have to take care that $\wtx_{k+n+r}$ differs from $\wtx_{k+n+r}'$ with a small probability,
given $\wtx_{k+n}=\wtx_{k+n}',\ldots,\wtx_{k+n+r-1}=\wtx_{k+n+r-1}$.
Condition~(A1) is intended to keep the probability of these undesired events small.
\bigskip

\section{Examples}
\label{S3}
\subsection{Linear Poisson-INGARCH processes}
\label{SS3.1}

In this section we discuss some of the most popular specifications for INGARCH(1,1) processes.
We begin with a linear INGARCH(1,1) process allowing for real-valued covariates, where
\be
\label{3.1}
\lambda_{t+1} \,=\, a_t \; Y_t \,+\, b_t \; \lambda_t \,+\, Z_t.
\end{equation}
Without covariates and with $a_t=a$, $b_t=b$ $\forall t$, this model has become popular for modeling count data.
\citet{RS00} proposed such a model for describing the number of trades on the New York Stock Exchange 
in certain time intervals and called it BIN(1,1) model.
Stationarity  and other properties for this model where derived by \citet{Str00},
\citet{FLO06} who referred to it as INGARCH(1,1) model, and \citet{FRT09}. 
\citet{ACKR16} generalized model (\ref{3.1}) by augmenting a covariate process 
and coined the term PARX (Poisson autoregression with exogeneous covariates).
These authors also proved the existence of a stationary distribution.
We study first the non-explosive case.

{\cor
\label{C3.1}
Suppose that
\begin{itemize}
\item[(i)] (\ref{2.1a}) is fulfilled,
\item[(ii)] (\ref{3.1}) holds, where $a_t,b_t\geq 0$ and $L_2=\sup\{a_t+b_t\colon\; t\in\N_0\}<1$,
\item[(iii)] $\sup\{\E Z_t\colon\; t\in\N_0\}<\infty$ and $Z_t$ is a non-negative random variable (covariate) which is independent of
$\lambda_0,Y_0,Z_0,\ldots,\lambda_{t-1},Y_{t-1},Z_{t-1},\lambda_t,Y_t$,
\item[(iv)]
$\E\lambda_0<\infty$.
\end{itemize}
Then the process $(X_t)_{t\in\N_0}$ is absolutely regular with coefficients
satisfying
\bd
\beta^X(n)
\,\leq\, L_2^{n-1} \; \frac{1}{1-L_1} \; M,
\ed
where $L_1=\sup\{b_t\colon\; t\in\N_0\}$ 
and $M=2(\E \lambda_0 + \sup\{\E Z_t\colon\; t\in\N_0\}/(1-L_2))$.
}
\bigskip{
\rem{As it can be seen from the proof, we obtain the same result if we consider more generally  $\lambda_{t+1} \,= g(\, a_t \; Y_t \,+\, b_t \; \lambda_t \,+\, Z_t)$ for some Lipschitz function $g$ with $\text{Lip}(g)\leq 1$ under conditions (i), (iii), and (iv) of Corollary~\ref{C3.1} if  $L_2=\sup\{|a_t|+|b_t|\colon\; t\in\N_0\}<1$. In particular, we obtain absolute regularity with an exponential rate for softplus INGARCH(1,1) processes with exogenous regressors under the conditions on the coefficients $a_t$ and $b_t$ and on the regressors $(Z_t)_t$ mentioned above. Softplus INGARCH processes without exogeneous regressors have been introduced just recently by \citet{WZH22}, 
	where $g=s_c$ is the so-called softplus function
	\begin{displaymath}
		s_c(x) \,=\, c\, \ln( 1 \,+\, e^{x/c}),\quad \text{with }c>0;
	\end{displaymath}
see also Section~4 for further details.}\label{rem2}
}

\bigskip

The proof of Corollary~\ref{C3.1} relies on the application of Theorem~\ref{T2.1} with the simple metric
$d(\lambda,\lambda')=|\lambda-\lambda'|$.
In case of an explosive INGARCH(1,1) process, however, it could well happen that this distance 
 is no longer appropriate.
To see this, consider the simple case of a specification
\bd
\lambda_{t+1} \,=\, a Y_t \,+\, C_t,
\ed
where $0<a<1$ and $C_t$ being an arbitrarily large non-negative constant.
Recall that our estimate (\ref{2.7}) of the local coefficients
of absolute regularity $\beta^X(k,n)$ contains the factor $\widetilde{\E}d(\wtl_{k+1},\wtl_{k+1}')$
which would be $\widetilde{\E}|\wtl_{k+1}-\wtl_{k+1}'|$ using the $L_1$-distance.
Let $((\wty_t,\wtl_t))_{t\in\N_0}$ and $((\wty_t',\wtl'_t))_{t\in\N_0}$ be two independent versions of the
bivariate process.
Then $\wtl_{t+1}-\wtl'_{t+1}=a(\wty_t-\wty'_t)$ and, conditioned on $\wtl_t$, $\wtl_t'$,
$\wty_t$ and $\wty_t'$ are independent and Poisson distributed with respective intensities $\wtl_t$ and $\wtl_t'$.
Since $\wtl_t,\wtl_t'\geq C_{t-1}$ it follows that
$\widetilde{\E}|\wtl_{t+1}-\wtl_{t+1}'|\rightarrow\infty$ as $C_{t-1}\to\infty$,
which means that assumption~(A3) will be violated.
We show that the alternative distance $|\sqrt{\lambda}-\sqrt{\lambda'}|$ saves the day.
The use of such a square root transformation should not come as a big surprise. 
Recall that $d_{TV}\left( \mbox{Pois}(\lambda), \mbox{Pois}(\lambda') \right) \,\leq\, d(\lambda, \lambda') 
\leq\sqrt{2/e}|\sqrt{\lambda}-\sqrt{\lambda'}|$.
On the other hand, it is well-known that a square root transformation on Poisson variates
has the effect of being variance-stabilizing. In fact, if $Y_\lambda\sim \mbox{Pois}(\lambda)$,
then $\E[(\sqrt{Y_\lambda}-\sqrt{\lambda})^2]\rightarrow 1/4$ as $\lambda\rightarrow\infty$;
see e.g.~\citet[p.~96]{MCN89}.
This transformation is similar to the Anscombe transform ($x\mapsto 2\sqrt{x+3/8}$) which is also a classical
tool to treat Poisson data.
On the other hand, for small values of $\lambda$ and $\lambda'$, the distance $|\lambda-\lambda'|$
turns out to be more suitable when a contraction property has to be derived; see the proof of Corollary~\ref{C3.2} below.
In view of this, we choose
\bea
\label{3.2}
d(\lambda, \lambda') 
& = & \left\{ \begin{array}{ll}
|\lambda-\lambda'|/M & \quad \mbox{ if } \sqrt{\lambda}+\sqrt{\lambda'} \leq M, \\
|\sqrt{\lambda}-\sqrt{\lambda'}| & \quad \mbox{ if } \sqrt{\lambda}+\sqrt{\lambda'} > M
\end{array} \right. \nonumber \\
& = & \min\left\{ |\lambda-\lambda'|/M, |\sqrt{\lambda}-\sqrt{\lambda'}| \right\},
\eea
where a suitable choice of the constant $M\in (0,\infty)$ becomes apparent from the proof of Corollary~\ref{C3.2} below.
\bigskip

{\cor
\label{C3.2}
Suppose that
\begin{equation}
\label{3.11}
\lambda_{t+1} \,=\, a_t \; Y_t \,+\, b_t \; \lambda_t \,+\, Z_t,
\end{equation}
where
\begin{itemize}
\item[(i)] $a_t,b_t\geq 0$ with $\sup\{a_t+b_t\colon \;t\in\N_0\}<1$,
\item[(ii)] $\sup\{ \E|\sqrt{Z_t}-\E\sqrt{Z_t}|\colon \; t\in\N_0\}<\infty$,
\item[(iii)] $\E\sqrt{\lambda_0}<\infty$.
\end{itemize}
Then the process $(X_t)_{t\in\N_0}$ is absolutely regular with coefficients
satisfying
\bd
\beta^X(n) \,=\, O\left( \rho^n \right)
\ed
for some $\rho<1$.
}
\bigskip

Note that the random variable $Z_t$ may get arbitrarily large as $t$ increases,
for example, it could represent a trend. Hence, we allow for nonstationary, explosive scenarios here.
\bigskip

\subsection{Log-linear Poisson-INGARCH processes}
\label{SS3.2}

Next, we consider the log-linear model proposed by \citet{FT11}.

{\prop
\label{P3.4}
Suppose that
\begin{equation}
\label{3.13}
\log(\lambda_{t+1}) \,=\, d \,+\, a \; \log(\lambda_t) \,+\, b \; \log(Y_t+1) \,+\, Z_t.
\end{equation}
where $d\in\R$ and $|a|+|b|<1$, and $(Z_t)_{t\in\N_0}$ are i.i.d.~random variables such that $\E|Z_0|<\infty$.

Then
\begin{itemize}
\item[(i)] there exists a (strictly) stationary version of $((Y_t,\lambda_t,Z_t))_t$,
\item[(ii)] if additionally $\E[e^{2Z_0}]<\infty$, then the process $(X_t)_t$ is absolutely regular with exponentially
decaying coefficients.
\end{itemize}
}
\bigskip

\subsection{Mixed Poisson-INGARCH processes}
\label{SS3.3}

{The above results can be generalized to models where the Poisson distribution
is replaced by certain mixed Poisson distributions.
We consider two cases, the zero-inflated Poisson and the negative binomial distribution,
in more details. In both cases, our model can be put in the above framework by setting
\be
\label{3.31}
\lambda_t \,=\, f_t(Y_{t-1},\lambda_{t-1},Z_{t-1})
\,:=\, Z_{t-1}^{(2)} \; \widetilde{f}_t(Y_{t-1},\lambda_{t-1},Z_{t-1}^{(1)}),
\end{equation}
where $Z_{t-1}=(Z_{t-1}^{(1)},Z_{t-1}^{(2)})$ is a covariate with independent components $Z^{(1)}_t$ and  $Z^{(2)}_t$,
$Z^{(2)}_t$ being non-negative.}

{If $(Z_t^{(2)})_t$ in \eqref{3.31} is a sequence of i.i.d.~$\mbox{Bin}(1,p)$ variables for some $p\in (0,1)$ and if 
\begin{displaymath}
Y_t\mid \F_{t-1}  \,\sim\, \mbox{Pois}(\lambda_t),
\end{displaymath}
with $\F_{t-1}=\sigma\big(Y_0,\lambda_0,Z_0,\dots,Y_{t-1},\lambda_{t-1}, Z_{t-1}\big)$
then, conditioned on\\ $\F_{t-1}^{(1)}=\sigma\big(Y_0,\lambda_0,Z_0,\dots,Y_{t-2},\lambda_{t-2}, Z_{t-2},Y_{t-1},\lambda_{t-1}, Z_{t-1}^{(1)}\big)$,
$Y_t$ has a zero-inflated Poisson distribution (see \citet{Lam92}) with parameters~$p$ and $\nu_t=\widetilde{f}_t(Y_{t-1},\lambda_{t-1},Z_{t-1}^{(1)})$, i.e.
\begin{displaymath}
P\big( Y_t=k \mid \F_{t-1}^{(1)} \big) \,=\, \left\{ \begin{array}{ll}
p \; e^{-\nu_t} \nu_t^k/k! & \quad \mbox{ if } k\geq 1, \\
(1-p) \,+\, p \; e^{-\nu_t} & \quad \mbox{ if } k=0.
\end{array} \right.
\end{displaymath}
Similar INGARCH models with such a distribution were considered e.g.~in \citet{Zhu11}
to account for overdispersion and potential extreme observations.}

{If instead $(Z_t^{(2)})_t$ has a Gamma distribution with parameters $a,b>0$ and
\begin{displaymath}
Y_t\mid \F_{t-1}  \,\sim\, \mbox{Pois}(\lambda_t),
\end{displaymath}
then, conditioned on $\F_{t-1}^{(1)}$ as above, $Y_t$ has a negative binomial distribution.
Indeed, since a \mbox{Gamma}$(a,b)$ distribution has a density~$p$ with
\begin{displaymath}
p(x) \,=\, \left\{ \begin{array}{ll}
\frac{b^{a}}{\Gamma(a)} \; x^{a-1} \; e^{-bx} & \quad \mbox{ if } x\geq 0, \\
0 & \quad \mbox{ if } x<0 \end{array} \right.
\end{displaymath}
we obtain that, for all $k\in\N_0$, 
\bean
P(Y_t=k \mid \F_{t-1}^{(1)})
& = & \int_0^\infty \frac{b^{a}}{\Gamma(a)} \; x^{a-1} \; e^{-bx} \;\; e^{-\lambda x} 
\; \frac{(\lambda x)^k}{k!} \, dx \\
& = & \frac{1}{\Gamma(a)\; k!} \; b^{a} \; \lambda^k \; \int_0^\infty x^{a+k-1} \; 
e^{-(\lambda+b)x} \, dx \\
& = & \frac{\Gamma(a+k)}{\Gamma(a)\; k!} \; \left( \frac{b}{\lambda+b} \right)^{a}
\; \left(1 \,-\, \frac{b}{\lambda+b} \right)^k.
\eean
This is the probability mass function of a $\mbox{NB}(a,b/(\lambda+b))$ distribution.}

{In both cases, we may use Theorem~\ref{T2.1} to prove that the process $(X_t)_t$ is absolutely regular 
with exponentially decaying coefficients. Note that under validity of (A3), it suffices to check (A1) and (A2) for $\widetilde f_t$ rather than $f_t$ (with $L_1, \, L_2< b/a$ for the NB example). To see this, consider a coupling such that $\widetilde Z_t^{(2)}=\widetilde Z_t^{(2)'}$ which then gives
$$
\begin{aligned}
&	\widetilde{E}\left( d(\widetilde Z_t^{(2)}\; \widetilde f_t(\widetilde{Y}_t,\widetilde{\lambda}_t,\widetilde{Z}^{(1)}_t),
\widetilde Z_t^{(2)}\; f(\widetilde{Y}_t',\widetilde{\lambda}_t',\widetilde Z^{(1)'}_t))
\mid \widetilde{\lambda}_t,\widetilde{\lambda}_t' \right)\\
&\,=\, \widetilde E[\widetilde Z_t^{(2)}] \; \widetilde{E}\left( d(f(\widetilde{Y}_t,\widetilde{\lambda}_t,\widetilde{Z}^{(1)}_t),
f(\widetilde{Y}_t',\widetilde{\lambda}_t',\widetilde{Z}^{(1)'}_t)) \mid \widetilde{\lambda}_t,\widetilde{\lambda}_t' \right).
\end{aligned}
$$
Those are the two most suitable cases for applications; anyway the distribution of other independent variables $Z^{(2)}_t$ for which  (A1) and (A2) hold can also be considered.  }
\bigskip

\section{Relation to previous work and possible perspectives}
\label{S4}

In the context of {\em stationary} INGARCH processes, absolute regularity with a geometric
decay of the mixing coefficients of the count process
has already been proved in \citet{Neu11} under a fully contractive condition,
\be
\label{4.1}
\left| f(y,\lambda) \,-\, f(y',\lambda') \right|
\,\leq\, a\; |y-y'| \,+\, b\; |\lambda-\lambda'| \qquad \forall y,y'\in\N_0, \forall \lambda,\lambda'\geq 0,
\end{equation}
where $a$ and $b$ are non-negative constants with $a+b<1$.
\citet{DN19} proved absolute regularity with a somewhat unusual subgeometric
decay of the coefficients for GARCH and INGARCH processes of arbitrary order~$p$ and~$q$
under a weaker semi-contractive condition,
\be
\label{4.2}
\big| f(y_1,\ldots,y_p;\lambda_1,\ldots,\lambda_q) \,-\, f(y_1,\ldots,y_p;\lambda_1',\ldots,\lambda_q') \big|
\,\leq\, \sum_{i=1}^q c_i | \lambda_i-\lambda_i' |
\end{equation}
for all $y_1,\ldots,y_p\in\N_0$; $\lambda_1,\ldots,\lambda_q,\lambda_1',\ldots,\lambda_q'\geq 0$, 
where $c_1,\ldots,c_q$ are non-negative constants with $c_1+\cdots +c_q<1$.

For the specification (\ref{3.11}) and without a covariate ($Z_t=0$ $\forall t$), conditions (\ref{4.1}) and (\ref{4.2}) are both
fulfilled. However, in case of a non-stationary covariate process $(Z_t)_{t\in\N_0}$,
stationarity of the process $((Y_t,\lambda_t))_{t\in\N_0}$ might fail and the results 
in the above mentioned papers cannot be used.
More seriously, in case of an explosive behavior, e.g. if $Z_t$ is non-random with $Z_t\rightarrow \infty$
as $t\to\infty$, the stability condition (2.5) in \citet{Neu11} as well as the
drift condition (A1) in \citet{DN19} are violated and a direct adaptation of the proofs
in those papers seems to be impossible.

In case of a specification $\lambda_t=(a\sqrt{Y_{t-1}}+b\sqrt{\lambda_{t-1}})^2$ we obtain that
\bd
\left| \wtl_{t+1} \,-\, \wtl_{t+1}' \right| 
\,=\, \left| a^2 (\wty_t-\wty_t') \,+\, b^2 (\wtl_t-\wtl_t') \,+\,
2ab \left( \sqrt{\wty_t} \sqrt{\wtl_t} \,-\, \sqrt{\wty_t'} \sqrt{\wtl_t'} \right) \right|.
\ed
If $\wty_t$ and $\wty_t'$ are equal but large, then the right-hand side of this equation will be
dominated by the term $2ab\sqrt{\wty_t}|\sqrt{\wtl_t}-\sqrt{\wtl_t'}|$ which shows that both
(\ref{4.1}) and (\ref{4.2}) are violated. However, Theorem~2.1 is applicable.
One can follow the lines of the proof of Corollary~\ref{C3.1} to verify the validity of (A1) to (A3)
for $d(\lambda, \lambda')=|\sqrt{\lambda}-\sqrt{\lambda'}|$.
\bigskip

{We would like to mention that similar results as in our paper are possible
for INGARCH models with distributions different from the Poisson.
\citet{DMN21} proved existence and uniqueness of a stationary distribution
and absolute regularity of the count process for models where the Poisson
distribution is replaced by the distribution of the difference of two independent
Poisson variates (special case of a Skellam distribution).
We expect similar results in the case of a generalized Poisson distribution
which was advocated in the context of INGARCH models in \citet{Zhu12}.
Moreover, standard GARCH models with a normal distribution can be treated by this approach as well.}

{After our paper was completed, a referee brought to our attention a recently accepted
paper by \citet{WZH22}, where the log-linear function is replaced by the  softplus function~$s_c$ stated in Remark~\ref{rem2}.
The corresponding Poisson-INGARCH model is specified by
\begin{displaymath}
\lambda_t \,=\, s_c\Big( \alpha_0 \,+\, \sum_{i=1}^p \alpha_i X_{t-i}
\,+\, \sum_{j=1}^q \beta_j \lambda_{t-j} \Big).
\end{displaymath}
These authors proved existence and uniqueness of a stationary distribution and,
relying on results derived by \citet{DN19}, absolute regularity of the count process 
with a \textit{sub}exponential decay rate for the corresponding mixing coefficients under weaker summability assumptions on the coefficients than in our Remark~\ref{rem2} (but neither allowing for exogenous regressors nor for time-varying coefficients).} The setting of general observation-driven models with covariates is also considered in \cite{DNT20}.
\bigskip

\section{Testing for a  trend in linear INARCH(1) models with application to COVID-19 data}
\label{S5}
\subsection{Statistical study}
Suppose that we observe $Y_0,\ldots,Y_n$ of a linear INARCH process as in (\ref{3.1}) with
$b_t=0,\;\forall t$. We aim to test stationarity versus the presence of an isotonic
trend. Thus, the null hypothesis will be that $\E Y_1=\cdots =\E Y_n$ while the
alternative can be characterized by $\E Y_1\leq\E Y_2\leq\cdots \leq\E Y_n$ with at least
one strict inequality in this chain of inequalities.
When we fit a linear model
\bd
Y_t \,=\, \theta_0 \,+\, \theta_1 \; t \,+\, \varepsilon_t, \qquad t=1,\ldots,n,
\ed
with a possibly non-stationary sequence of  innovations $(\varepsilon_t)_t$,
then the null hypothesis corresponds to $\theta_1=0$ and the alternative to $\theta_1>0$.
(Even if the above linear model is not adequate, a projection will lead to $\theta_1>0$.)
The following discussion will be simplified when we change over to
an orthogonal regression model,
\bd
Y_t \,=\, \theta_0 \,+\, \theta_1 \; w_t
\,+\, \varepsilon_t, \qquad t=1,\ldots,n,
\ed
where $w_t=(t-\frac{n+1}{2})/\sqrt{\sum_{s=1}^n (s-\frac{n+1}{2})^2}$.
Then the columns in the corresponding design matrix are orthogonal and the $l_2$ norm
of the vector composed of the entries in the second column is equal to 1.
Therefore, we obtain for the least squares estimator $\widehat{\theta}_1$ of $\theta_1$
that
\bd
\widehat{\theta}_1 \,=\, \sum_{t=1}^n w_t Y_t,
\ed
As before, we have $\theta_1=\E \widehat{\theta}_1>0$ if there is any positive (linear or nonlinear)
trend and $\theta_1=0$ under the null hypothesis.
Therefore, $\widehat{\theta}_1$ can be used as a test statistic.
\bigskip

\begin{prop}\label{t1.ex1}
Suppose that $Y_0,\dots, Y_n$ is a stretch of observations of a stationary INARCH(1) with constant coefficients such that
$$
\lambda_t=aY_{t-1}+b_0\quad\text{with } a\in(0,1),~b_0\geq 0.
$$
Then, with  $\sigma^2=b_0/(1-a)^3$,
$$
\widehat{\theta}_1 \,\stackrel{d}{\longrightarrow}\, Z_0 \sim \mathcal N(0,\sigma^2).
$$ 
\end{prop}
\bigskip

We show that the test statistic is asymptotically unbounded for the special case
of a linear trend component in the intensity function. Other situations such as general
polynomial trends can be treated similarly.
\bigskip

 \begin{prop}\label{t2.ex1}
Suppose that $Y_0,\dots, Y_n$ is a stretch of observations of a nonstationary INARCH(1) with constant coefficients and trend such that
$$
\lambda_t=aY_{t-1}+b_0+ b_1 t\quad\text{with } a\in(0,1),~b_0\geq 0 \text{ and } b_1>0,\quad t\in\N_0
$$
and $\lambda_0$ has a  finite absolute fourth moment.  Then, for any $K>0$
$$
P(\widehat{\theta}_1>K)\ninfty 1.
$$
\end{prop}
\bigskip

Hence, a test rejecting the null if 
$$
\widehat{\theta}_1/\sigma>z_{1-\alpha}
$$
is asymptotically of size $\alpha$ and consistent. Here, $z_{1-\alpha}$ denotes the
$(1-\alpha)$ quantile of $\mathcal N(0,1)$. In practice, $\sigma$ is unknown and has
to be estimated consistently. For our simulations and the data example
presented below, we used the corresponding OLS-estimators  $\widehat a$ and $\widehat b_0$
to obtain $\widehat \sigma^2=\widehat b_0/(1-\widehat a)^3$.
More precisely, we considered the model stated in Proposition~\ref{t2.ex1} and calculated
	\begin{equation}\label{eq.ols}
		\begin{aligned}
	\begin{pmatrix}
	\widehat a\\
	\widehat b_0\\
	\widehat b_1
	\end{pmatrix}=(X^TX)^{-1}\, X^TY\qquad \text{with}\quad
 Y=	\begin{pmatrix}
 	Y_1\\
 	\vdots\\
 	Y_n
 \end{pmatrix}\qquad \text{and}\quad
 X=	\begin{pmatrix}
Y_0	&1&1\\
	\vdots&	\vdots&	\vdots\\
Y_{n-1}&1&n
\end{pmatrix}
\end{aligned}
\end{equation}
\begin{lem}\label{l.ols}
In the situation of Proposition~\ref{t1.ex1} the OLS estimators of $a$ and $b_0$ are consistent.
\end{lem}
\medskip
With similar arguments as in the proof of Lemma~5.1 it can be shown that the OLS estimators of $a$ and $b_0$ are also consistent under the alternative described in Proposition~\ref{t2.ex1}.

{
\rem{We stick to the INARCH(1) model although a  generalization of Propositions 5.1 and 5.2 to INGARCH(1,1) models is possible.
In the latter case, the naive OLS estimation is no longer feasible since the intensity process is unobserved.
Of course, there are consistent estimators for stationary INGARCH(1,1) processes as well.
However, their behavior under the alternative would have to be investigated, too. This goes far beyond the scope of the paper.} }
\subsection{Numerical study}
Next, we investigate the finite sample behavior of the proposed test. Considering low and moderate levels of persistence ($a=0.2$ and $a=0.5$) we increase the effect of a linear trend from $b_1=0$ (null hypothesis) to $b_1=0.1$ holding the intercept fixed ($b_0=1$). We vary the sample size $n=50,\; 100,\;200$ for $a=0.2$ and $n=  100,\;250,\; 500$ for $a= 0.5$. The results for $\alpha=0.1$  using 5000 Monte Carlo loops are displayed in Figure~\ref{sims}. The power properties of our test are very convincing however it tends to reject a true null too often in small samples. In particular, note that for $a=0.5$  increasing the sample size from 250 to 500 improves the performance of the test under the null but barely influences the behavior under the alternative (the solid and dashed line in Figure~\ref{sims} nearly coincide).
\begin{figure}
	\includegraphics[width=15cm]{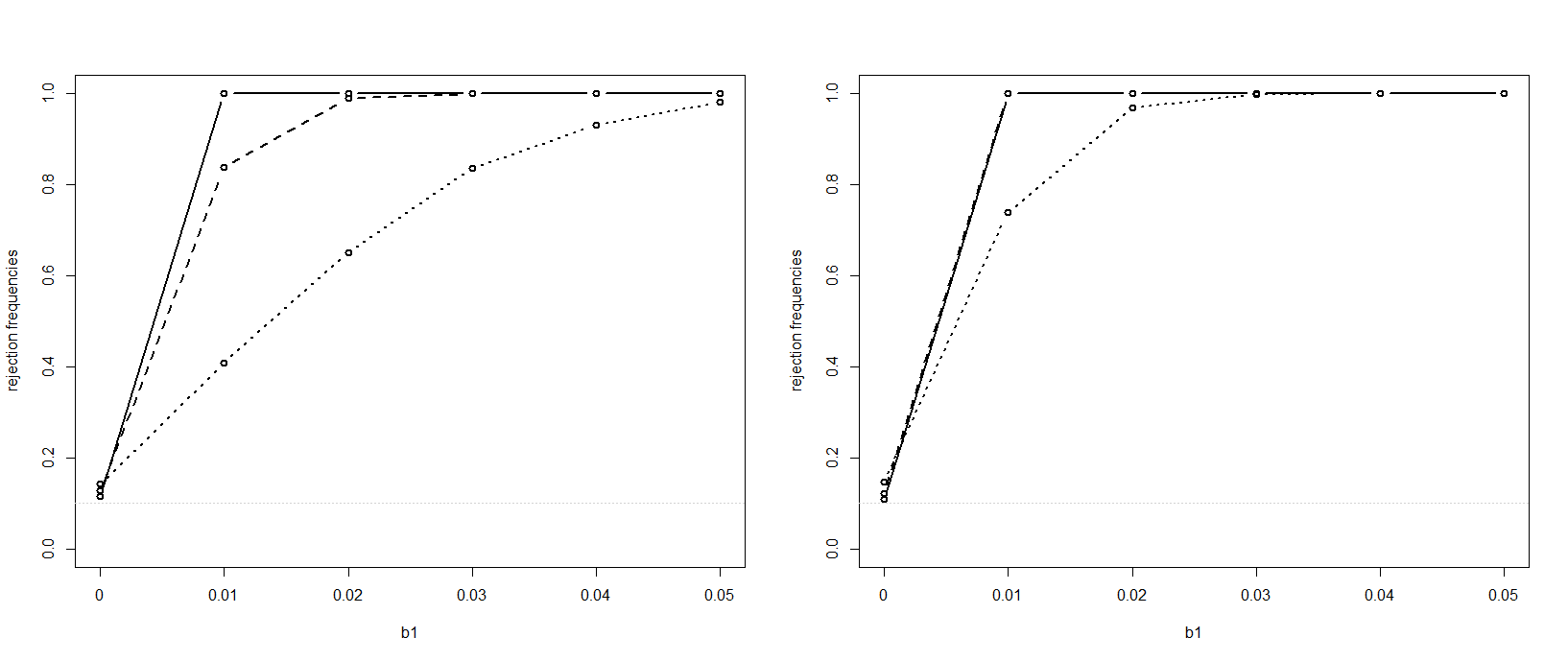}
	\caption{\scriptsize Left: $a=0.2$ and $n=50$ dotted, $n=100$ dashed, $n=200$ solid line; ~
		Right: $a=0.5$ and $n=100$ dotted, $n=250$ dashed, $n=500$ solid line.}\label{sims}
\end{figure}
\subsection{Analysis of COVID-19 data}
We applied our test to investigate daily COVID-19 infection numbers as well as the cases of deaths related to  COVID-19 in France and Germany from July $15$ to September $15$, 2020 using a data set published by the \citet{ECDPC}; see Figure~\ref{covid}. \textcolor{red}{Observing a weekly periodicity in the data, we pre-processed the data eliminating an estimated seasonal component.} Obviously, no test is required to observe an increasing trend in the daily infection numbers in France as well as in Germany. Our test clearly rejects the null in both cases (France: $\widehat\theta_1/\widehat\sigma=349693.30$, Germany: $\widehat\theta_1/\widehat\sigma=96.39$). However, the situation changes if we look at the cases of deaths. Again, the null is rejected for France ($\widehat\theta_1/\widehat\sigma=7.50$)  at any reasonable level. Contrary, evaluating the test statistic based on the number of deaths in Germany that are related to COVID-19, we obtain $\widehat\theta_1/\widehat\sigma=-0.11$, that is, the null hypothesis of no trend is not rejected at any reasonable level. 
We also studied a shift of the window of observation of 16 days, i.e.~we considered the period from  August 1 to September 30. Then, unfortunately, the null is rejected for both countries for the number of daily infections as well as for the COVID-19 related number of deaths. 
\begin{figure}
\includegraphics[width=15cm]{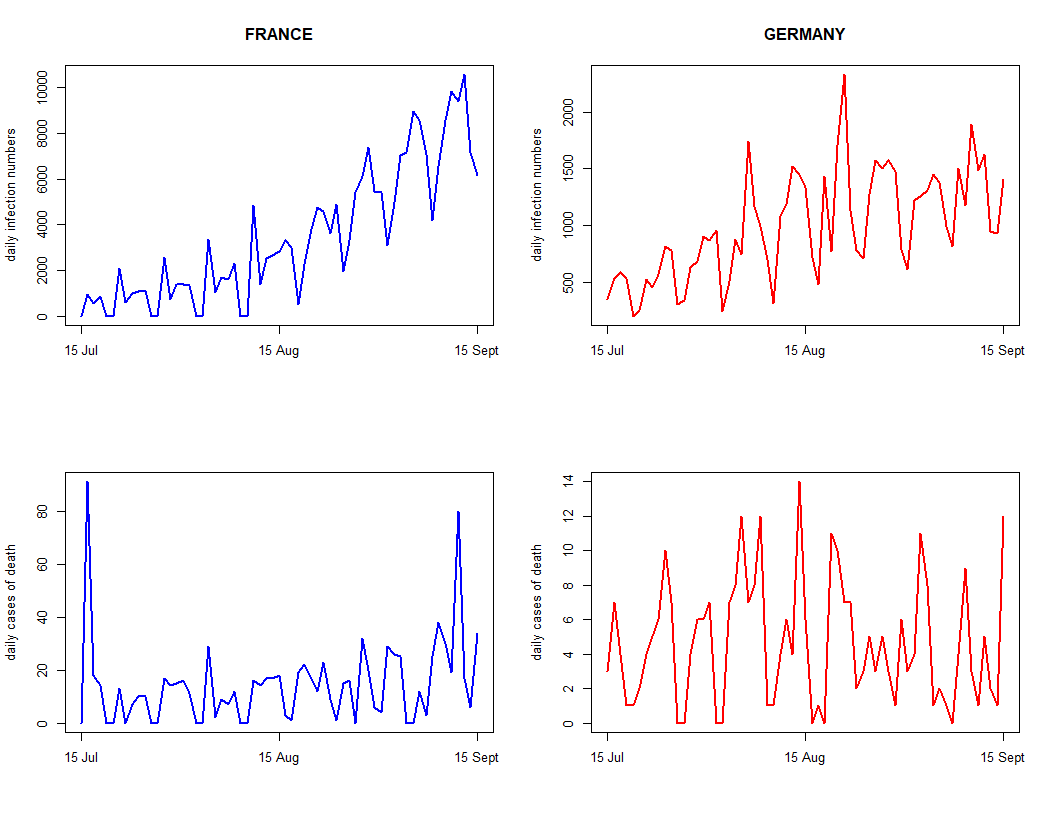}
\caption{\scriptsize Daily COVID-19 infection numbers (top) as well as the cases of deaths (bottom) related to  COVID-19 in France (left, blue) and Germany (right, red) from July $15$ to September $15$, 2020.}\label{covid}
\end{figure}

\bigskip

\section{Proofs }
\label{S6}

\begin{proof}[Proof of Theorem~\ref{T2.1}]
	The proof of assertion (i) is given in the running text of Section~\ref{S2}.
	
	To prove (ii), we first identify a function~$g$, which will satisfy the
	required equality (\ref{2.11}). We consider backward iterations $g^{[k]}$, 
	where (with $x_i=(y_i,z_i)$) $g^{[1]}(x_1,\lambda_1):=f(y_1,\lambda_1,z_1)$
	and, for $k\geq 2$, $g^{[k]}(x_1,\ldots,x_k,\lambda_k):=f(y_1,g^{[k-1]}(x_2,\ldots,x_k,\lambda_k),z_1)$.
Using the  idea of iterations as  in exercise 46 of \cite{Dou18}, we consider and we set precise approximations to $\lambda_t$,
\bd
	\lambda_t^{[k]} \,=\, g^{[k]}(X_{t-1},\ldots,X_{t-k},\bar{\lambda}),
	\ed
	where $\bar{\lambda}=\E \lambda_0$.
	It follows from (A2) that
	\bd
	\E\left[ d(\lambda_t, \lambda_t^{[k]}) \right]
	\,\leq\, L_2^k \; \E\left[ d(\lambda_{t-k},\bar{\lambda}) \right]
	\,=\, L_2^k \; \E\left[ d(\lambda_0,\bar{\lambda}) \right],
	\ed
	which implies that
	\bd
	d( \lambda_t, \lambda_t^{[k]} ) \,\stackrel{\P}{\longrightarrow}\, 0 \qquad \mbox{ as } k\to\infty.
	\ed
	This implies
	\bd
	d_{TV}\left( \mbox{Pois}(\lambda_t), \mbox{Pois}(\lambda_t^{[k]}) \right) \,\stackrel{\P}{\longrightarrow}\, 0
	\ed
	and, therefore,
	\bd
	\lambda_t \,-\, \lambda_t^{[k]} \,\stackrel{\P}{\longrightarrow}\, 0.
	\ed
	By taking an appropriate subsequence $(k_n)_{n\in\N}$ of $\N$ we even obtain
	\be
	\label{pt21.1}
	\lambda_t \,-\, \lambda_t^{[k_n]} \,\stackrel{a.s.}{\longrightarrow}\, 0.
\end{equation}
In order to obtain a well-defined function~$g$, we define, for any sequence $x_1,x_2,\ldots$,
\bd
g(x_1,x_2,\ldots) \,=\, \limsup_{n\to\infty} g^{[k_n]}(x_1,\ldots,x_{k_n},\bar{\lambda}).
\ed
As a limit of the measurable functions $g^{[k_n]}$, $g$ is also $(\sigma(\mathcal{Z})-\mathcal{B})$-measurable.
{}From (\ref{pt21.1}) we conclude that
\bd
\lambda_t \,=\, \lim_{n\to\infty} \lambda_t^{[k_n]} \,=\, g(X_{t-1},X_{t-2},\ldots)
\ed
holds with probability~1, as required.

Since absolute regularity of the process $(X_t)_{t\in\Z}$ implies strong mixing (see e.g.~\citet[p.~20]{Dou94})
we conclude from Remark~2.6 on page~50 in combination with Proposition~2.8 on page~51 in \citet{Bra07}
that any stationary version of this process is also ergodic.
Finally, we conclude from (\ref{pt21.1})
by proposition~2.10(ii) in \citet[p.~54]{Bra07} that also
the process $((Y_t,\lambda_t,Z_t))_{t\in\Z}$ is ergodic.
\end{proof}
\bigskip

\begin{proof}[Proof of Corollary~\ref{C3.1}]
	We choose the distance~$d$ as $d(\lambda,\lambda')=|\lambda-\lambda'|$ and verify
	that conditions (A1) to (A3) are fulfilled.
	\begin{itemize}
		\item[{\em (A1):}]
		We construct the coupling such that $\wtz_t=\wtz_t'$. Then
		\bd
		|\wtl_{t+1}-\wtl_{t+1}'| \; \1(\wty_t=\wty_t') \,\leq\, b_t \; |\wtl_t-\wtl_t'|.
		\ed
		Therefore, (A1) is fulfilled with $L_1=\sup\{b_t\colon\; t\in\N_0\}$.
		\item[{\em (A2):}]
		We couple the covariates such that $\widetilde{Z}_t=\widetilde{Z}_t'$.
		The count variables are coupled in such a way that $\wty_t\geq\wty_t'$ if $\wtl_t\geq\wtl_t'$
		and $\wty_t\leq\wty_t'$ if $\wtl_t\leq\wtl_t'$. 
		Such a coupling is necessary and sufficient for
		$\widetilde{\E}(|\wty_t-\wty_t'|\big|\wtl_t,\wtl_t')=|\wtl_t-\wtl_t'|$;
		otherwise the term on the left-hand side will be larger.
		Note that the maximal coupling   but also the simple ``additive coupling''  share  this property.
		The latter can be constructed as follows.	If $\wtl_t'\leq\wtl_t$ then $\wty_t=\wty_t'+W_t$, where $W_t\sim\mbox{Pois}(\wtl_t-\wtl_t')$
		is independent of $\wty_t'$.
		Vice versa, if $\wtl_t'>\wtl_t$ then $\wty_t'=\wty_t+W_t$, where $W_t\sim\mbox{Pois}(\wtl_t'-\wtl_t)$
		is independent of $\wty_t$.
		Then
		\bean
		\hspace{1.5cm}\widetilde{\E}\left( |\wtl_{t+1}-\wtl_{t+1}'| \Big| \wtl_t,\wtl_t' \right)
		  =  a_t \; \widetilde{\E}\left( |\wty_t-\wty_t'| \Big| \wtl_t,\wtl_t' \right)
		\,+\, b_t\; |\wtl_t-\wtl_t'|  
	  =   (a_t+b_t)\; |\wtl_t-\wtl_t'|,
		\eean
		that is, (A2) is fulfilled with $L_2=\sup\{a_t+b_t\colon\; t\in\N_0\}$.
	\end{itemize}
	It follows from (\ref{3.1}) that $\E \lambda_{k+1}\leq (a_k+b_k)\E \lambda_k + \E Z_k$,
	which implies that
	\bd
	\E \lambda_k \,\leq\, \E \lambda_0 \,+\, \frac{1}{1-L_2} \;\sup\{\E Z_t\colon\; t\in\N_0\}.
	\ed 
\end{proof}
\bigskip

\begin{proof}[Proof of Corollary~\ref{C3.2}]
	\begin{itemize}
		\item[{\em (A1):}]
		We construct the coupling such that $\wtz_t=\wtz_t'$.
		Since $|\sqrt{\lambda+c}-\sqrt{\lambda'+c}|\leq |\sqrt{\lambda}-\sqrt{\lambda'}|$
		holds for all $\lambda,\lambda',c\geq 0$ we obtain that
		\bd
		\left| \sqrt{\wtl_{t+1}} - \sqrt{\wtl_{t+1}'} \right| \; \1( \wty_t = \wty_t' )
		\,\leq\, \sqrt{b_t} \; \left| \sqrt{\wtl_t} - \sqrt{\wtl_t'} \right|
		\ed
		On the other hand, the inequality
		$|\wtl_{t+1}-\wtl_{t+1}'|\; \1( \wty_t = \wty_t' )\leq b\;|\wtl_t-\wtl_t'|$ is obvious.
		Hence, (A1) is fulfilled with $L_1=\sup\{\sqrt{b_t}\colon\, t\in\N_0\}$.
		\item[{\em (A2):}]
		We couple the covariates such that $\widetilde{Z}_t=\widetilde{Z}_t'$.
		For the count variables, we use an additive coupling as described in the proof of Corollary~\ref{C3.1}(A2).
		This yields in particular that $\wty_t\geq\wty_t'$ if $\wtl_t\geq\wtl_t'$
		and $\wty_t\leq\wty_t'$ if $\wtl_t\leq\wtl_t'$.
		We will show that, for some $\rho<1$, 
		\be
		\label{pc32.0}
		\widetilde{\E}\big( d(\wtl_{t+1}, \wtl_{t+1}') \mid \wtl_t, \wtl_t' \big)
		\,\leq\, \rho \; d\left(\wtl_t, \wtl_t' \right),
	\end{equation}
	provided that the constant $M$ in (\ref{3.2}) is chosen appropriately.
	To this end, we distinguish between two cases:
	\begin{itemize}
		\item[Case (i):] $\sqrt{\wtl_t}+\sqrt{\wtl_t'}\leq M$
	\end{itemize}
	Then $d(\wtl_t,\wtl_t')=|\wtl_t-\wtl_t'|/M$ and it follows that
	\bea
	\label{pc32.1}
	\lefteqn{ \widetilde{\E}\left( d(\wtl_{t+1},\wtl_{t+1}') \mid \wtl_t,\wtl_t' \right) } \nonumber \\
	& \leq & \widetilde{\E}\left( |\wtl_{t+1}-\wtl_{t+1}'|/M \mid \wtl_t,\wtl_t' \right) \nonumber \\
	& = & (a_t+b_t) \; |\wtl_t-\wtl_t'|/M \,=\, (a_t+b_t) \; d(\wtl_t, \wtl_t').
	\eea
	
	\begin{itemize}
		\item[Case (ii):] $\sqrt{\wtl_t}+\sqrt{\wtl_t'}> M$
	\end{itemize}
	In this case, $d(\wtl_t,\wtl_t')=|\sqrt{\wtl_t}-\sqrt{\wtl_t'}|$.
	We choose $\epsilon>0$ such that $\sup\{\sqrt{a_t+b_t}\colon\; t\in\N_0\}<1-\epsilon$.
	To simplify notation, let $\lambda,\lambda'$ be non-random with $\lambda\geq\lambda'$, $\sqrt{\lambda}+\sqrt{\lambda'}>M$ and let
	$Y=Y'+Z$, where $Y'\sim \mbox{Pois}(\lambda')$ and $Z\sim \mbox{Pois}(\lambda-\lambda')$ are independent.
	Furthermore, we drop the index~$t$ with~$a_t$ and $b_t$.
	Again, we have to distinguish between two cases.
	\begin{itemize}
		\item[a):] $\epsilon\sqrt{\lambda}\geq (1+\epsilon)\sqrt{\lambda'}$
	\end{itemize}
	In this case the proof of (\ref{pc32.0}) is almost trivial.
	We have
	\bea
	\label{pc32.2}
	\lefteqn{ \E\left[ \sqrt{aY+b\lambda} \,-\, \sqrt{aY'+b\lambda'} \right] } \nonumber \\
	& \leq & \E \sqrt{ aY+b\lambda } \nonumber \\
	& \leq & \sqrt{a+b} \; \sqrt{\lambda} \nonumber \\
	& = & \sqrt{a+b} \; (1+\epsilon) \; \left(1-\frac{\epsilon}{1+\epsilon}\right) \; \sqrt{\lambda} \nonumber \\
	& \leq & \sqrt{a+b} \; (1+\epsilon) \; (\sqrt{\lambda}-\sqrt{\lambda'}) \nonumber \\
	& \leq & \frac{\sqrt{a+b}}{1-\epsilon} \; |\sqrt{\lambda}-\sqrt{\lambda'}|.
	\eea
	Here, the second inequality follows by Jensen's inequality since $x\mapsto \sqrt{x}$ is a concave function.
	
	\begin{itemize}
		\item[b):] $\epsilon\sqrt{\lambda}< (1+\epsilon)\sqrt{\lambda'}$
	\end{itemize}
	This case requires more effort. We split up
	\bea
	\label{pc32.11}
	\lefteqn{ \E\left[ \sqrt{aY+b\lambda} \,-\, \sqrt{aY'+b\lambda'} \right] } \nonumber \\
	& \leq & \E\left[ \left(\sqrt{aY+b\lambda} \,-\, \sqrt{aY'+b\lambda'}\right)) \;
	\1\left( \sqrt{aY+b\lambda}+\sqrt{aY'+b\lambda'} \geq (1-\epsilon)\;\sqrt{a+b}(\sqrt{\lambda}+\sqrt{\lambda'}) \right) 
	\right] \nonumber \\
	& & {} + \E\left[ \left(\sqrt{aY+b\lambda} \,-\, \sqrt{aY'+b\lambda'}\right) \;
	\1\left( \sqrt{aY+b\lambda} < (1-\epsilon)\;\sqrt{(a+b)\lambda} \mbox{ and }
	\sqrt{aY'+b\lambda'} \geq (1-\epsilon)\;\sqrt{(a+b)\lambda'} \right) \right] \nonumber \\
	& & {} + \E\left[ \left(\sqrt{aY+b\lambda} \,-\, \sqrt{aY'+b\lambda'}\right) \;
	\1\left( \sqrt{aY'+b\lambda'} < (1-\epsilon)\;\sqrt{(a+b)\lambda'} \right) \right] \nonumber \\
	& =: & T_1 \,+\, T_2 \,+\, T_3,
	\eea
	say.
	Then
	\bea
	\label{pc32.12}
	T_1 & = & \E\left[ \frac{aY+b\lambda-aY'-b\lambda'}{ \sqrt{aY+b\lambda}+\sqrt{aY'+b\lambda'} } \;
	\1\left( \sqrt{aY+b\lambda}+\sqrt{aY'+b\lambda'} \geq (1-\epsilon)\;\sqrt{a+b}(\sqrt{\lambda}+\sqrt{\lambda'}) \right) 
	\right] \nonumber \\
	& \leq & \E\left[ \frac{aY+b\lambda-aY'-b\lambda'}{ (1-\epsilon) \; \sqrt{a+b} \; (\sqrt{\lambda}+\sqrt{\lambda'}) } \right]
	\nonumber \\
	& = & \frac{ \sqrt{a+b} }{ 1-\epsilon } \; \frac{ \lambda-\lambda' }{ \sqrt{\lambda}+\sqrt{\lambda'} }
	\,=\, \frac{ \sqrt{a+b} }{ 1-\epsilon } \; \left| \sqrt{\lambda} \,-\, \sqrt{\lambda'} \right|.
	\eea
	Since $\sqrt{aY+b\lambda}<(1-\epsilon)\sqrt{(a+b)\lambda}$ implies that $Y< (1-\epsilon)^2\lambda$,
	and therefore $|Y-\lambda|> (1-(1-\epsilon)^2)\lambda$, we obtain that
	\bea
	\label{pc32.13}
	T_2 & \leq & \E\left[ (1-\epsilon) \; \sqrt{a+b} \; \left( \sqrt{\lambda} - \sqrt{\lambda'} \right) \;
	\1\left( |Y-\lambda| \geq (1-(1-\epsilon)^2) \; \lambda \right) \right] \nonumber \\
	& \leq & (1-\epsilon) \; \sqrt{a+b} \; \left( \sqrt{\lambda} - \sqrt{\lambda'} \right)
	\frac{1}{ (1-(1-\epsilon)^2)^2 \; \lambda } \nonumber \\
	& \leq & \left| \sqrt{\lambda} - \sqrt{\lambda'} \right| \;
	\frac{ (1-\epsilon) \; \sqrt{a+b} }{ (1-(1-\epsilon)^2)^2 } \; \frac{4}{M^2}.
	\eea
	Note that the last inequality follows from $2\sqrt{\lambda}\geq \sqrt{\lambda}+\sqrt{\lambda'}>M$.
	To estimate $T_3$, we use the simple estimates
	\bean
	\sqrt{aY+b\lambda} \,-\, \sqrt{aY'+b\lambda'}
	& = & \sqrt{aY+b\lambda} \,-\, \sqrt{aY'+b\lambda} \,+\, \sqrt{aY'+b\lambda} \,-\, \sqrt{aY'+b\lambda'} \\
	& \leq & \sqrt{a} \; \left( \sqrt{Y} - \sqrt{Y'} \right)
	\,+\, \sqrt{b} \; \left( \sqrt{\lambda} \,-\, \sqrt{\lambda'} \right)
	\eean
	and $\sqrt{Y}-\sqrt{Y'}\leq Y-Y'$, as well as the fact that
	$\sqrt{aY'+b\lambda'}<(1-\epsilon)\sqrt{(a+b)\lambda'}$ implies
	that $|Y'-\lambda'|>(1-(1-\epsilon^2))\lambda'$. This leads to
	\bean
	T_3 & \leq & \sqrt{a} \; \E\left[ \left(\sqrt{Y}-\sqrt{Y'}\right) \;
	\1(\sqrt{aY'+b\lambda'}<(1-\epsilon)\sqrt{(a+b)\lambda'}) \right] \\
	& & {} \,+\, \sqrt{b} \; \E\left[ \left(\sqrt{\lambda}-\sqrt{\lambda'}\right) \;
	\1(\sqrt{aY'+b\lambda'}<(1-\epsilon)\sqrt{(a+b)\lambda'}) \right] \\
	& \leq & \sqrt{a} \; \E\left[ (Y-Y') \; \1(|Y'-\lambda'|>(1-(1-\epsilon^2))\lambda') \right] \\
	& & {} \,+\, \sqrt{b} \; \left(\sqrt{\lambda}-\sqrt{\lambda'}\right) \; \P(|Y'-\lambda'|>(1-(1-\epsilon^2))\lambda') \\
	& \leq & \left( \sqrt{a} \; (\lambda-\lambda') \,+\, \sqrt{b} \; (\sqrt{\lambda}-\sqrt{\lambda'}) \right)
	\; \frac{ \E(Y'-\lambda')^2 }{ (1-(1-\epsilon^2))^2 \; {\lambda'}^2 }.
	\eean
	{}From $\epsilon\sqrt{\lambda} < (1+\epsilon)\sqrt{\lambda'}$ we obtain that
	$M\leq \sqrt{\lambda}+\sqrt{\lambda'}\leq \frac{1+2\epsilon}{\epsilon} \sqrt{\lambda'}$,
	which leads to
	\be
	\label{pc32.14}
	T_3 \,\leq\, \left| \sqrt{\lambda} \,-\, \sqrt{\lambda'} \right| \; 
	\frac{1}{(1-(1-\epsilon^2))^2} \; \left( \frac{1+2\epsilon}{\epsilon} \right)^2
	\; \left( \frac{\sqrt{a}}{M} \,+\, \frac{\sqrt{b}}{M^2} \right).
\end{equation}

To sum up, we conclude from (\ref{pc32.1}) to (\ref{pc32.14}) that (\ref{pc32.0})
is fulfilled for
\bd
\rho \,=\, \frac{ \sqrt{a+b} }{ 1-\epsilon } \,+\,
\frac{ (1-\epsilon) \; \sqrt{a+b} }{ (1-(1-\epsilon)^2)^2 } \; \frac{4}{M^2} \,+\,
\frac{1}{(1-(1-\epsilon^2))^2} \; \left( \frac{1+2\epsilon}{\epsilon} \right)^2
\; \left( \frac{\sqrt{a}}{M} \,+\, \frac{\sqrt{b}}{M^2} \right).
\ed
Choosing now the constant $M$ sufficiently large we obtain that $\rho<1$,
as required.
\item[{\em (A3):}]
Part (i) of (A3) is fulfilled by assumption.

Assume that the processes $((\wty_t,\wtl_t,\wtz_t))_{t\in\N_0}$ and $((\wty_t',\wtl_t',\wtz_t'))_{t\in\N_0}$
are independent copies of the original process $((Y_t,\lambda_t,Z_t))_{t\in\N_0}$.
We have that
\bea
\label{pc32.21}
\lefteqn{ \left| \sqrt{\wtl_{t+1}} \,-\, \sqrt{\wtl_{t+1}'} \right| } \nonumber \\
& = & \left| \sqrt{ a\wty_t + b\wtl_t + \widetilde{Z}_t } \,-\, \sqrt{ a\wty_t' + b\wtl_t' + \widetilde{Z}_t} 
\,+\, \sqrt{ a\wty_t' + b\wtl_t' + \widetilde{Z}_t } \,-\, \sqrt{ a\wty_t' + b\wtl_t' + \widetilde{Z}_t'} \right| \nonumber \\
& \leq & \left| \sqrt{ a\wty_t + b\wtl_t } \,-\, \sqrt{ a\wty_t' + b\wty_t' } \right| 
\,+\, \left| \sqrt{\widetilde{Z}_t} \,-\, \sqrt{\widetilde{Z}_t'} \right| \nonumber \\
& \leq & \left| \sqrt{ a\wty_t + b\wtl_t } \,-\, \sqrt{ (a+b)\wtl_t } \right| \nonumber \\
& & {} \,+\, \sqrt{a+b} \; \left| \sqrt{\wtl_t} \,-\, \sqrt{\wtl_t'} \right| \nonumber \\
& & {} \,+\, \left| \sqrt{ a\wty_t' + b\wtl_t' } \,-\, \sqrt{ (a+b)\wtl_t' } \right| \nonumber \\
& & {} \,+\, \left| \sqrt{\widetilde{Z}_t} \,-\, \sqrt{\widetilde{Z}_t'} \right| \nonumber \\
& =: & R_{t,1} \,+\, \cdots \,+\, R_{t,4},
\eea
say. 
We obtain that
\bea
\label{pc32.22}
\widetilde{\E}\left( R_{t,1} \mid \wtl_t,\wtl_t' \right)
& = & \widetilde{\E}\left( \frac{ a\; |\wty_t-\wtl_t| }{ \sqrt{a\wty_t+b\wtl_t} \,+\, \sqrt{(a+b)\wtl_t} }
\Big| \wtl_t,\wtl_t' \right) \nonumber \\
& \leq & \frac{ a }{ \sqrt{a+b} } \; \widetilde{\E}\left( |\wty_t-\wtl_t|/\sqrt{\wtl_t} \big| \wtl_t,\wtl_t' \right) \nonumber \\
& \leq & \frac{ a }{ \sqrt{a+b} } \; \sqrt{ \widetilde{\E}\left( (\wty_t-\wtl_t)^2/\wtl_t \big| \wtl_t,\wtl_t' \right) }
\,=\, \frac{ a }{ \sqrt{a+b} }
\eea
and, for the same reason,
\be
\label{pc32.23}
\widetilde{\E}\left( R_{t,3} \mid \wtl_t,\wtl_t' \right) \,\leq\, \frac{ a }{ \sqrt{a+b} }.
\end{equation}
Finally, we have that
\be
\label{pc32.24}
\widetilde{\E}\left( \big| \sqrt{\widetilde{Z}_t} \,-\, \sqrt{\widetilde{Z}_t'} \big| \Big| \wtl_t,\wtl_t' \right)
\,=\, \widetilde{\E} \left| \sqrt{\widetilde{Z}_t} \,-\, \sqrt{\widetilde{Z}_t'} \right| 
\,\leq\, 2 \; \E\left| \sqrt{Z_t} \,-\, \E\sqrt{Z_t} \right|.
\end{equation}
It follows from (\ref{pc32.21}) to (\ref{pc32.24}) that part (ii) of condition (A3) is fulfilled with $L_3=\sqrt{a+b}$ and
$M_0=2\; \sup\{\E|\sqrt{Z_t}-\E\sqrt{Z_t}|\colon\; t\in \N_0\}\,+\,2a/\sqrt{a+b}$.
\end{itemize}
\end{proof}
\bigskip

\begin{proof}[Proof of Proposition~\ref{P3.4}] \ 
\begin{itemize}
\item[(i)]
First of all, note that the process $(V_t)_{t\in\N_0}$ with $V_t=(\log(\lambda_t),\log(Y_t+1),Z_t)$
forms a time-homogeneous Markov chain. Let  $S=\R\times\log(\N)\times\R$ be the state space of this process.

In order to derive a contraction property, we choose the metric
\begin{displaymath}
\Delta\Big( (x,y,z), (x',y',z') \big)
\,=\, \kappa_1 |x-x'| \,+\, \kappa_2 |y-y'| \,+\, |z-z'|,
\end{displaymath}
where $\kappa_1$ and $\kappa_2$ are strictly positive constants such that $|a|\leq\kappa_1$,
$|b|\leq\kappa_2$, and $\kappa:=\kappa_1+\kappa_2<1$.
We show that we can couple two versions of the process $(V_t)_{t\in\N_0}$,
$(\widetilde{V}_t)_{t\in\N_0}$ and $(\widetilde{V}_t')_{t\in\N_0}$, such that
\begin{equation}
\label{pp34.1}
\widetilde{\E} \left( \left. \Delta\big( \widetilde{V}_{t+1}, \widetilde{V}_{t+1}' \big)
\right| \widetilde{V}_t, \widetilde{V}_t' \right)
\,\leq\, \kappa \; \Delta\big( \widetilde{V}_t, \widetilde{V}_t' \big).
\end{equation}
We couple the corresponding covariate processes such that they coincide,
i.e. $\wtz_t=\wtz_t'$ $\forall t\in\N_0$.
Let $v=(x,y,z),v'=(x',y',z')\in S$ be arbitrary.
We assume that $\widetilde{V}_t=v$ and $\widetilde{V}_t'=v'$ and construct
$\widetilde{V}_{t+1}=(\log(\wtl_{t+1}),\log(\wty_{t+1}+1),\wtz_{t+1})$
and $\widetilde{V}_{t+1}'=(\log(\wtl_{t+1}'),\log(\wty_{t+1}'+1),\wtz_{t+1}')$ 
as follows.
According to the model equation (\ref{3.13}) we set
\begin{displaymath}
\log(\wtl_{t+1}) \,=\, d \,+\, a y \,+\, b x \,+\, \widetilde{Z}_t
\end{displaymath}
and
\begin{displaymath}
\log(\wtl_{t+1}') \,=\, d \,+\, a y' \,+\, b x' \,+\, \widetilde{Z}_t'.
\end{displaymath}
Conditioned on $\widetilde{V}_t$ and $\widetilde{V}_t'$, the random variables $\wty_{t+1}$ and $\wty_{t+1}'$
have to follow Poisson distributions with intensities $\wtl_{t+1}$ and $\wtl_{t+1}'$, respectively.
At this point we employ a coupling such that $\wty_{t+1}-\wty_{t+1}'$ has with probability~1
the same sign as $\wtl_{t+1}-\wtl_{t+1}'$. 
This implies in particular that
\begin{eqnarray}
\label{pp34.2}
\lefteqn{ \widetilde{\E} \Big( \big| \log(\wty_{t+1}+1) \,-\, \log(\wty_{t+1}'+1) \big| \Big|
\widetilde{V}_t, \widetilde{V}_t' \Big) } \nonumber \\
& = & \Big| \widetilde{\E} \left( \left. \log(\wty_{t+1}+1) \,-\, \log(\wty_{t+1}'+1) \right|
\widetilde{V}_t, \widetilde{V}_t' \right) \Big| \nonumber \\
& = & \Big| \E\left( \left. \log(Y_{t+1}+1) \right| \lambda_{t+1}=\wtl_{t+1} \right)
\,-\, \E\left( \left. \log(Y_{t+1}+1) \right| \lambda_{t+1}=\wtl_{t+1}' \right) \Big|.
\end{eqnarray}
To estimate the term on the right-hand side of (\ref{pp34.2}), we show that, for $Y^{(\lambda)}\sim\mbox{Pois}(\lambda)$,
\be
\label{pp34.3}
\frac{d}{d\lambda} \left\{ \E \log(Y^{(\lambda)}+1)) \right\} \,\leq\, \frac{1}{\lambda} \qquad \forall \lambda>0.
\end{equation}
To see this, suppose that $Y^{(\lambda)}\sim \mbox{Pois}(\lambda)$ and $Y^{(\epsilon)}\sim \mbox{Pois}(\epsilon)$
are independent. Then
\bean
\lefteqn{ \E\left[ \log(Y^{(\lambda)}+Y^{(\epsilon)}+1) \,-\, \log(Y^{(\lambda)}+1) \right] } \\
& = & e^{-\epsilon} \; \epsilon \; \sum_{k=0}^\infty [\log(k+2) - \log(k+1)] \; e^{-\lambda} \frac{\lambda^k}{k!} \\
& & {} \,+\, e^{-\epsilon} \; \sum_{l=2}^\infty \frac{\epsilon^l}{l!} \;
\sum_{k=0}^\infty [\log(k+l+1) - \log(k+1)] \; e^{-\lambda} \frac{\lambda^k}{k!} \\
& =: & T_{\epsilon,1} \,+\, T_{\epsilon,2},
\eean
say.
Since $\log(k+l+1)-\log(k+1)=\int_{k+1}^{k+l+1} \frac{1}{u} \, du\leq \frac{l}{k+1}$ we obtain that
\bd
0 \,\leq\, T_{\epsilon,1} \,=\, e^{-\epsilon} \; \epsilon \sum_{k=0}^\infty \frac{1}{k+1} \; e^{-\lambda}
\frac{\lambda^k}{k!}
\,\leq\, \epsilon \; \frac{1}{\lambda} \; \sum_{k=0}^\infty e^{-\lambda} \frac{\lambda^{k+1}}{(k+1)!}
\,=\, \frac{\epsilon}{\lambda} \; \P( Y_\lambda \neq 0 )
\ed
as well as
\bd
0 \,\leq\, T_{\epsilon,2} \,\leq\, \frac{1}{\lambda} \; \sum_{l=2}^\infty \frac{\epsilon^\lambda}{l!}
\; \sum_{k=0}^\infty e^{-\lambda} \frac{\lambda^{k+1}}{(k+1)!}
\,\leq\, \frac{1}{\lambda} \; \sum_{l=2}^\infty \epsilon^l \,=\, \frac{1}{\lambda} \; \frac{\epsilon^2}{1-\epsilon}.
\ed
Therefore,
\bd
\frac{d}{d\lambda} \E[\log( Y^{(\lambda)}+1)]
\,=\, \lim_{\epsilon\to 0} \frac{T_{\epsilon,1}}{\epsilon} \,=\, \frac{\P(Y^{(\lambda)}\neq 0)}{\lambda}
\,\leq\, \frac{1}{\lambda},
\ed
that is, (\ref{pp34.3}) holds true.
Hence, we obtain from (\ref{pp34.2}) that
\begin{equation}
\label{pp34.3a}
\widetilde{\E}\Big( 
\big| \log(\wty_{t+1}+1) \,-\, \log(\wty_{t+1}'+1) \big| \Big| \wty_t,\wty_t' \Big)
\,\leq\, \Big| \log(\wtl_{t+1}) \,-\, \log(\wtl_{t+1}') \Big|. 
\end{equation}
Recall that we have, by construction, $\wtz_{t+1}=\wtz_{t+1}'$. Using this and the above calculations we obtain
\begin{eqnarray}
\label{pp34.4}
\lefteqn{ \widetilde{\E}\Big( \Delta\big( \widetilde{V}_{t+1}, \widetilde{V}_{t+1}' \big) \Big|
\widetilde{V}_t, \widetilde{V}_t' \Big) } \nonumber \\
& \leq & \kappa_1 \big| \log(\wtl_{t+1}) \,-\, \log(\wtl_{t+1}') \big|
\,+\, \kappa_2 \widetilde{\E}\Big( 
\big| \log(\wty_{t+1}+1) \,-\, \log(\wty_{t+1}'+1) \big| \Big| \widetilde{V}_t,\widetilde{V}_t' \Big) \nonumber \\
& \leq & \kappa \big| \log(\wtl_{t+1}) \,-\, \log(\wtl_{t+1}') \big| \nonumber \\
& \leq & \kappa \Big( |a| \big| \log(\wtl_t) \,-\, \log(\wtl_t') \big| 
\,+\, |b| \big| \log(\wty_t+1) \,-\, \log(\wty_t'+1) \big| \Big) \nonumber \\
& \leq & \kappa \Delta\big( \widetilde{V}_t, \widetilde{V}_t' \big).
\end{eqnarray}
It remains to translate this contraction property for random variables into
a contraction property for the corresponding distributions.
For the metric $\Delta$ on~$S$, we define
\begin{displaymath}
{\mathcal P}(S) \,=\, \big\{Q\colon\quad Q \mbox{ is a probability distribution on } S
\mbox{ with } \int \Delta(z_0,z)\, dQ(z)<\infty \big\},
\end{displaymath}
where $z_0\in S$ is arbitrary.
For two probability measures~$Q,Q'\in {\mathcal P}(S)$, we define the Kantorovich distance
based on the metric $\Delta$ (also known as Wasserstein $L^1$ distance)
by
\bd
{\mathcal K}(Q, Q') \,:=\, \inf_{V\sim Q, V'\sim Q'} \widetilde{\E} \Delta(V, V'),
\ed
where the infimum is taken over all random variables~$V$ and~$V'$ defined on a common
probability space $(\widetilde{\Omega},\widetilde{\F},\widetilde{P})$ with respective laws~$Q$ and~$Q'$.
We denote the Markov kernel of the processes $(V_t)_{t\in\N_0}$ by $\pi^V$.
Now we obtain immediately from (\ref{pp34.4}) that
\begin{equation}
\label{pp34.5}
{\mathcal K}(Q \pi^V, Q' \pi^V) \,\leq\, \kappa \, {\mathcal K}(Q, Q').
\end{equation}
The space ${\mathcal P}(S)$ equipped with the Kantorovich metric ${\mathcal K}$ is complete.
Since by (\ref{pp34.5}) the mapping $\pi^V$ is contractive it follows by the Banach
fixed point theorem that the Markov kernel $\pi^V$ admits a unique fixed point $Q^V$,
i.e. $Q^V\pi^V=Q^V$. In other words, $Q^V$ is the unique stationary distribution of the process
$(V_t)_{t\in\N_0}$. Therefore, the process $((Y_t,\lambda_t,Z_t))_{t\in\N_0}$ has a unique stationary
distribution as well.

\item[(ii)]

	In this case, we do not use Theorem~\ref{T2.1} to prove absolute regularity, but Proposition~\ref{p.mixing}.
	To this end, we make use of a contraction property on the logarithmic scale and
	change over to the square root scale afterwards.
	As above, we construct on a suitable probability space $(\widetilde{\Omega},\widetilde{\F},\widetilde{\P})$
	two versions of the three-dimensional process, $((\wty_t,\wtl_t,\wtz_t))_{t\in\N_0}$ and
	$((\wty_t',\wtl_t',\wtz_t'))_{t\in\N_0}$ where these two processes evolve independently up to time~$k$.
	Then $\wtl_{k+1}$ and $\wtl_{k+1}'$ are independent, as required.
	For $t=k+1,\ldots,k+n-1$, we couple these processes such that $\wtz_t=\wtz_t'$
	as well as $\wty_t\geq\wty_t'$ if $\wtl_t\geq\wtl_t'$ and vice versa $\wty_t\leq\wty_t'$ if $\wtl_t\leq\wtl_t'$.

We obtain from (\ref{pp34.3a}) that
\begin{eqnarray*}
\lefteqn{ \widetilde{\E} \Big( |\log(\wtl_{t+1}) \,-\, \log(\wtl_{t+1}')| \big| \wtl_t,\wtl_t' \Big) } \\
& \leq & a \; \big| \log(\wtl_t) \,-\, \log(\wtl_t') \big|
\,+\, b \; \widetilde{\E}\Big( | \log(\wty_t+1) \,-\, \log(\wty_{t+1}'+1) | \big| \wtl_t,\wtl_t' \Big) \\
& \leq & (a+b) \; \big| \log(\wtl_t) \,-\, \log(\wtl_t') \big|
\end{eqnarray*}
holds for all $t\in\{k+1,\ldots\}$. Using this inequality $(n-1)$-times we obtain that
	\be
	\label{pp34.11}
	\E\left( |\log(\wtl_{k+n}) - \log(\wtl_{k+n}')| \Big| \wtl_{k+1},\wtl_{k+1}' \right)
	\,\leq\, (|a|+|b|)^{n-1} \; \left|\log(\wtl_{k+1}) - \log(\wtl_{k+1}')\right|.
\end{equation}

For $t=k+n,k+n+1,\ldots$, we use a maximal coupling of the count variables, that is,
\bd
\widetilde{\P}\left( \wty_t\neq\wty_t' \Big| \wtl_t, \wtl_t' \right)
\,=\, d_{TV}\left( \mbox{Pois}(\wtl_t), \mbox{Pois}(\wtl_t') \right).
\ed
This implies by Proposition~\ref{p.mixing} that 
\bea
\label{pp34.5a}
\lefteqn{ \beta^X(k,n) } \nonumber \\
& = & \widetilde{\P}\left( \wty_{k+n}\neq \wty_{k+n}' \Big| \wtl_{k+1}, \wtl_{k+1}' \right) \nonumber \\
& & {} \,+\,
 \sum_{r=1}^\infty \widetilde{\P}\left( \wty_{k+n+r}\neq \wty_{k+n+r}',
\wty_{k+n+r-1}=\wty_{k+n+r-1}',\ldots, \wty_{k+n}=\wty_{k+n}' \Big| \wtl_{k+1}, \wtl_{k+1}' \right)
\nonumber \\
& = & \sum_{r=0}^\infty \widetilde{\E}\left( d_{TV}\left( \mbox{Pois}(\wtl_{k+n+r}), \mbox{Pois}(\wtl_{k+n+r}') \right)
\Big| \wtl_{k+1}, \wtl_{k+1}' \right).
\eea
Finally, it remains to make the transition from our estimates of $|\log(\wtl_t)-\log(\wtl_t')|$
to the above total variation distances.
Since $x\mapsto e^{x/2}$ is a convex function we have, for $0\leq x\leq y$,
$|e^{x/2}-e^{y/2}|=\int_{x/2}^{y/2} e^{u/2}/2\, du \leq \frac{e^{x/2}+e^{y/2}}{8} |x-y|$, which implies that
\be
\label{pp34.6}
\left| \sqrt{\wtl_{k+n+r}} \,-\, \sqrt{\wtl_{k+n+r}'} \right|
\,\leq\, \frac{\sqrt{\wtl_{k+n+r}}+\sqrt{\wtl_{k+n+r}'}}{8}\; \left| \log(\wtl_{k+n+r}) - \log(\wtl_{k+n+r}') \right|.
\end{equation}
Using this and the estimate
$d_{TV}(\mbox{Pois}(\lambda),\mbox{Pois}(\lambda'))\leq \sqrt{2/e}|\sqrt{\lambda}-\sqrt{\lambda'}|$ we obtain
\bea
\label{pp34.7}
\widetilde{\P}\left( \wty_{k+n}\neq \wty_{k+n}' \right)
& = & \widetilde{\E} \left[ d_{TV}\left( \mbox{Pois}(\wtl_{k+n}), \mbox{Pois}(\wtl_{k+n}') \right) \right]
\nonumber \\
& \leq & \sqrt{ \frac{2}{e} } \; \widetilde{\E}\left[ \left| \sqrt{\wtl_{k+n}} \,-\, \sqrt{\wtl_{k+n}'} \right|
\right] \nonumber \\
& \leq & \sqrt{ \frac{2}{e} } \; \sqrt{ \widetilde{\E} \left( 
\left( \sqrt{\wtl_{k+n}}+\sqrt{\wtl_{k+n}'} \right)/8 \right)^2 } \; 
\sqrt{ \widetilde{\E} \left( \log(\wtl_{k+n}) \,-\, \log(\wtl_{k+n}') \right)^2 } \nonumber \\
& \leq & \sqrt{ \frac{1}{2e} } \; \sqrt{ \E[ \lambda_{k+n}^2 ] }
\; (|a|+|b|)^{n-1} \; \sqrt{ \widetilde{\E} \left( \log(\wtl_{k+1}) - \log(\wtl_{k+1}') \right)^2 }
\eea
and, analogously,
\bea
\label{pp34.8}
\lefteqn{ \widetilde{\P}\left( \wty_{k+n+r}\neq \wty_{k+n+r}',
\wty_{k+n+r-1}=\wty_{k+n+r-1}',\ldots, \wty_{k+n}=\wty_{k+n}' \Big| \wtl_{k+1}, \wtl_{k+1}' \right) }
\nonumber \\
& \leq & \sqrt{ \frac{1}{2e} } \; \sqrt{ \E[ \lambda_{k+n}^2 ] }
\; (|a|+|b|)^{n-1} \; |a|^l \; \sqrt{ \widetilde{\E} \left( \log(\wtl_{k+1}) - \log(\wtl_{k+1}') \right)^2 }.
\eea
It remains to show that $\E[\lambda_{k+n}^2]$ is bounded. If $Y\sim\mbox{Pois}(\lambda)$, then
$E[(Y+1)^2]=\lambda^2+3\lambda+1$.
This implies 
\begin{displaymath}
\E\big( \lambda_{t+1}^2 \big| \lambda_t \big)
\,=\, e^{2d}\; \E[e^{2Z_0}]\; \lambda_t^{2a}\, [(\lambda_t+2)(\lambda_t+1)]^b
\,\leq\, C_1\Big( \lambda_t^{2(a+b)} \,+\, 1 \Big),
\end{displaymath}
for some $C_1<\infty$.
Therefore we obtain that
\begin{displaymath}
\E\big( \lambda_{t+1}^2 \big| \lambda_t \big)
\,\leq\, C_0 \lambda_t^2 \,+\, C_2,
\end{displaymath}
for appropriate $C_0<1$ and $C_2<\infty$. From this recursion we conclude that
$\E[\lambda_{k+n}^2]$ is bounded.
(\ref{pp34.7}) and (\ref{pp34.8}) yield  that
\bd
\sup\{\beta^X(k,n)\colon\; k\in\N_0\}
\,=\, O\left( (|a|+|b|)^{n-1} \; \sum_{r=0}^\infty |a|^r \right) \,=\, O\left( (|a|+|b|)^n \right).
\ed
\end{itemize}
\end{proof} 
\bigskip

\begin{proof}[Proof of Proposition~\ref{t1.ex1}]
First, note that the contraction condition $a\in(0,1)$ assures existence of a strictly stationary version
of the process with $\beta$-mixing coefficients tending to zero at a geometric rate (see Corollary~\ref{C3.1} and Theorem~2.1 in \citet{Neu11}).
(Alternatively, since we are in the stationary case, Theorem~3.1 in Neumann (2011) containing both results.)
Moreover, all moments of $Y_t$ are finite, see e.g.~\citet[Example 4.1.6]{W18}.
Asymptotic normality of $\widehat \theta_1$ can be deduced from Application~1 in \citet{R95} setting $a_{i,n}=w_i$ and $\xi_i=Y_i-EY_i$
if $\sigma^2=\lim_{n\to\infty} \var(\widehat\theta_1)>0$.   To this end, note that
	from $\sum_{t=1}^n w_t=0$ and stationarity, we get
	$$
	\widehat\theta_1=\sum_{t=1}^n w_t (Y_t-EY_t).
	$$ 
	Additionally, straight-forward calculations yield
	$$
	s_n:=\frac{1}{n^3}\sum_{t=1}^n\left(t-\frac{n+1}{2}\right)^2=\frac{1}{12} +o(1).
	$$ 
	From \citet[Example 4.1.6]{W18} we know that 
	$$
	\cov(Y_0,Y_h)=a^{h}\, \frac{b_0}{(1-a)^2(1+a)}
	$$
	which gives
\begin{eqnarray*}	
\lefteqn{	\sigma^2\cdot \frac{(1-a)^2(1+a)}{b_0} } \\
& = & 1 \;+\; \lim_{n\to\infty}\,\frac{2}{s_n\,n^3}\, \sum_{t=2 }^n\left(t-\frac{n+1}{2}\right)\,a^t\,\sum_{s=1}^{t-1}\left(s-\frac{n+1}{2}\right)\,a^{-s} \\
& = & 1 \;+\; \lim_{n\to\infty}\,\frac{2}{s_n\,n^3}\, \sum_{t=2 }^n \left(t-\frac{n+1}{2}\right)\,
a^t\,\left[\frac{(t-1)a^{-(t+1)}-ta^{-t}+a^{-1}}{(a^{-1}-1)^2}-\frac{n+1}{2}\frac{a^{-t}-a^{-1}}{a^{-1}-1}\right] \\
&	= & 1 \;+\; \lim_{n\to\infty}\,\frac{2a}{s_n\,n^3(1-a)}\, \sum_{t=2 }^n \left(t-\frac{n+1}{2}\right)^2 \\
& = & 1 \;+\; \frac{2a}{1-a}
\end{eqnarray*}
	and finally yields the desired result.
\end{proof}
\medskip

\begin{proof}[Proof of Proposition~\ref{t2.ex1}]
We split up
\begin{equation}
\label{consist.1}
\widehat\theta_1=\sum_{t=1}^n w_t (Y_t-EY_t)+\sum_{t=1}^n w_t\, EY_t. 
\end{equation}
First, note that the second sum tends to infinity.  To see this, rewrite
$$
E Y_t=aEY_{t-1}+b_0+ b_1 t=\cdots=a^{t}EY_0+\sum_{k=0}^{t-1} a^k (b_0+b_1(t-k)).
$$
As $\sqrt{\sum_{s=1}^n (s-\frac{n+1}{2})^2}\geq C_1 n^{3/2} $, we obtain $\sup |w_t|\leq C_2n^{-1/2}$ which implies
\begin{eqnarray*}
\sum_{t=1}^n{w_t}\,EY_t
& = & o(n) \;+\; b_1\,\sum_{t=1}^n{w_t}\sum_{k=0}^{t-1}a^k \,t \\
& = & o(n) \;+\; b_1\,\sum_{t=1}^n w_t\, t\,\frac{a^t-1}{a-1} \\
& = & o(n) \;+\; \frac{b}{1-a}\sum_{t=1}^n t\,w_t \\
& =& C_3\, n^{3/2} \;+\; o(n^{3/2}),
\end{eqnarray*}
for some positive, finite constants $C_1,\; C_2,\; C_3$. 
It remains to show that the first sum in  \eqref{consist.1} is $o_P(n^{3/2})$. To this end, we consider
\begin{displaymath}
E\left[\sum_{t=1}^n w_t (Y_t-EY_t)\right]^2
\,\leq\, \frac{1}{ n} \sum_{h=-(n-1)}^{n-1}\sqrt{\beta^X(|h|)}\,\sum_{s=\max\{1,1-h\}}^{\min\{n,n-h\}}\, \sqrt[4]{E(Y_{s+h}-EY_{s+h})^4\,E(Y_s-EY_s)^4}
\end{displaymath}
applying the covariance inequality for $\alpha$-mixing processes in \cite{Dou94}, Theorem 3 (1), or Theorem 1.1 in \cite{Rio17},
the fact that the $\alpha$-mixing coefficients can be bounded from above by the corresponding $\beta$-mixing coefficients and Corollary \ref{C3.2}.
Recall that the $2^\text{nd}$ and the $3^\text{rd}$ central moment of a Pois($\lambda$) distributed random variable is just~$\lambda$
while the fourth central moment is $\lambda^2+3\lambda$. Using the binomial theorem and $E\lambda_s^2=O(s^2)$, we can further bound
\begin{eqnarray*}
\lefteqn{ E(Y_s-EY_s)^4 } \\
& = & E[\lambda_s^2+3\lambda_s]+4E[\lambda_s(\lambda_s-EY_s)]+6E[\lambda_s(\lambda_s-EY_s)^2]+E(\lambda_s-EY_s)^4 \\
& = & E[\lambda_s^2+3\lambda_s]+4aE[\lambda_s(Y_{s-1} -EY_{s-1})]+6a^2E[\lambda_s(Y_{s-1} -EY_{s-1})^2]+a^4E(Y_{s-1} -EY_{s-1})^4 \\
& = & E[\lambda_s^2+3\lambda_s]+4a^2E(Y_{s-1} -EY_{s-1})^2+6a^2E[\lambda_s(Y_{s-1} -EY_{s-1})^2]+a^4E(Y_{s-1} -EY_{s-1})^4 \\
& = & E[\lambda_s^2+3\lambda_s]+a^2[4+6E\lambda_s]E(Y_{s-1} -EY_{s-1})^2+6a^3E(Y_{s-1} -EY_{s-1})^3+a^4E(Y_{s-1} -EY_{s-1})^4 \\
& \leq & \widetilde Cs^2+a^2\,\widetilde Cs\,E(Y_{s-1} -EY_{s-1})^2+6a^3E(Y_{s-1} -EY_{s-1})^3+a^4E(Y_{s-1} -EY_{s-1})^4 \\
& \leq & \bar C s^2+6a^3E(Y_{s-1} -EY_{s-1})^3+a^4E(Y_{s-1} -EY_{s-1})^4 \\
& \leq & \bar C s^2+6a^3[C's^2+aE(Y_{s-2} -EY_{s-2})^3]+a^4E(Y_{s-1} -EY_{s-1})^4 \\
& \leq & C''s^2+a^4E(Y_{s-1} -EY_{s-1})^4.
\end{eqnarray*}
Iterating these calculations yields that $E(Y_s-EY_s)^4=O(s^2)$ which concludes the proof.
\end{proof}

\begin{proof}[Proof of Lemma~\ref{l.ols}]
Rewrite $Y_t=a Y_{t-1}+b_0+\eta_t$ with $\eta_t=Y_t-\lambda_t,~t=1,\dots,n,$. Using the corresponding matrix notation
and the definition of $X$, we have to show that $(X^TX)^{-1}X^T\eta=o_P(1)$, where $\eta=(\eta_1,\dots, \eta_n)^T$.
We proceed in two steps. First, we show that $NX^T\eta=o_P(1)$ with $N=diag(n^{-1 }, n^{- 1}, n^{-2 })$. Second, we show that $(NX^TX)^{-1}=O_P(1).$\\
For the first part, straight forward calculations show that
$$
NX^T\eta=	\frac{1}{n} \sum_{t=1}^n\begin{pmatrix}
Y_{t-1}(Y_t-\lambda_t)\\
	Y_t-\lambda_t\\
	 (Y_t-\lambda_t) t/n
\end{pmatrix}=o_P(1).
$$	
For the second part, we rewrite
$(NX^TX)^{-1}= M\, (NX^TXM)^{-1}$ with $M=diag(1, 1,n^{-1})$ and show that $NX^TXM$ converges stochastically to an invertible matrix. To this end, note that
\begin{eqnarray*}
NX^TXM
& = & \frac{1}{n}\,\begin{pmatrix}
\sum_{t=0}^{n-1} Y_t^2 & \sum_{t=0}^{n-1} Y_t&n^{-1}\sum_{t=0}^{n-1} t\,Y_t \\
\sum_{t=0}^{n-1}   Y_t&n&  (n+1)/2\\
n^{-1}\sum_{t=0}^{n-1} t\,Y_t&   (n+1)/2& (n+1)(2n+1)/(6n)
\end{pmatrix} \\
& = &
 \,\begin{pmatrix}
EY_0^2 & EY_0&EY_0/2 \\
EY_0&1&  1/2\\
EY_0/2&   1/2& 1/3
\end{pmatrix} \,+\, o_P(1).
\end{eqnarray*}
due to the exponentially decaying autocovariance function  of $(Y_t)_t$. Finally, straight forward calculations
show that the determinant of the remaining matrix is positive which concludes the proof.
\end{proof}

\

\paragraph{\bf Acknowledgment. }
This work was funded by CY Initiative of Excellence (grant ``Investissements d'Avenir'' ANR-16-IDEX-0008) 
Project ``EcoDep'' PSI-AAP2020-0000000013 (first and third authors) and within the MME-DII center of excellence (ANR-11-LABEX-0023-01),
and the Friedrich Schiller University in Jena (for the first author). We thank two anonymous referees
for their valuable comments that led to a significant improvement of the paper.

\bibliographystyle{harvard}

\end{document}